\documentclass[11pt]{amsart}
\usepackage{mathrsfs} 
\usepackage[margin=1.25in]{geometry}
\usepackage{ifthen}
\usepackage{graphicx}
\usepackage{epsfig}
\usepackage{dsfont, amssymb}
\usepackage{amsfonts,dsfont,amssymb,amsthm,stmaryrd,bbm}
\usepackage{amsmath}
\usepackage{verbatim}

\usepackage{xcolor}

\usepackage[toc,page]{appendix} 
\usepackage{hyperref,mathtools}
\hypersetup{colorlinks=true,pdftitle="",pdftex}

\let\originallesssim\lesssim
\let\originalgtrsim\gtrsim

\DeclareRobustCommand{\lesssim}{%
  \mathrel{\mathpalette\lowersim\originallesssim}%
}
\DeclareRobustCommand{\gtrsim}{%
  \mathrel{\mathpalette\lowersim\originalgtrsim}%
}

\makeatletter
\newcommand{\lowersim}[2]{%
  \sbox\z@{$#1<$}%
  \raisebox{-\dimexpr\height-\ht\z@}{$\m@th#1#2$}%
}
\makeatother

\newtheorem{conj}{Conjecture}[section]
\newtheorem{thm}{Theorem}[section]

\newtheorem{remark}[thm]{Remark}
\newtheorem{lem}[thm]{Lemma}
\newtheorem{prop}[thm]{Proposition}

\newtheorem{ques}[conj]{Question}

\newtheorem{defn}[thm]{Definition}
\newtheorem{cor}[thm]{Corollary}

\newcommand\independent{\protect\mathpalette{\protect\independent}{\perp}} 
\def\independent#1#2{\mathrel{\rlap{$#1#2$}\mkern2mu{#1#2}}}

\newcommand{\mZ}{\mathbb{Z}}

\def\Var{{\rm Var}}

\def\phi{\varphi}

\def\bee{\begin{eqnarray*}}
\def\ene{\end{eqnarray*}}
\usepackage{mathtools}

\DeclarePairedDelimiter\floor{\lfloor}{\rfloor}

\begin{document}

\title{Bernoulli Sums and R\'enyi Entropy Inequalities}

\author{Mokshay Madiman, James Melbourne,  and Cyril Roberto}

\address{University of Delaware, Department of Mathematical Sciences, 501 Ewing Hall, Newark DE 19716, USA.}
\email{madiman@udel.edu}

\address{Department of Probability and Statistics, Centro de Investigación en
matemáticas (CIMAT)}
\email{james.melbourne@cimat.mx}

\address{MODAL'X, UPL, Univ. Paris Nanterre, CNRS, F92000 Nanterre France}
\email{croberto@math.cnrs.fr}

\thanks{The last author was supported by the Labex MME-DII funded by ANR, reference ANR-11-LBX-0023-01 and ANR-15-CE40-0020-03 - LSD - Large Stochastic Dynamics, and the grant of the Simone and Cino Del Duca Foundation, France.}

\begin{abstract}
We investigate the R\'enyi entropy of  sums of {independent} integer-valued random variables through {Fourier} theoretic means, and give sharp comparisons between the variance and the R\'enyi entropy for sums of independent Bernoulli random variables.   As applications, we prove that a discrete ``min-entropy power'' is {superadditive with respect to convolution modulo} 
%on independent {random} variables 
a universal constant, and give new bounds on an entropic generalization of the Littlewood-Offord problem that are sharp in the ``Poisson regime''.
\end{abstract}

\keywords{entropy, variance, log-concave random variables}

\subjclass{94A17 (Primary), 60E15 (Secondary)}

\maketitle

\section{Introduction}

For a countable set $A$, $|A|$ will denote its cardinality.  The notation $\mathbb{P}$ will be reserved for a probability measure and the probability of an event $A$ will be denoted $\mathbb{P}(A)$.  For a discrete random variable $X$, { with values in a countable set $\mathcal{X}$,}
%on a countable set $\mathcal{X}$, 
we will denote its density function with respect to the counting measure as $f_X$ so that $f_X(x) = \mathbb{P}(X=x)$, $x \in \mathcal{X}$ (when $\mathcal{X}=\mathbb{Z}$, we may use the notation $p_n:=f_X(n)$ for simplicity).
We will denote for $f$ a function on a countable set $\mathcal{X}$, and $\alpha \in (0,\infty)$
\begin{align*}
    \|f\|_\alpha \coloneqq \left( \sum_{x \in \mathcal{X}} |f|^\alpha(i) \right)^{\frac 1 \alpha}.
\end{align*}
By continuous limits we define $\|f\|_\infty \coloneqq \sup_{x \in \mathcal{X}} |f(x)|$, and $\|f\|_0 = |\{x \in \mathcal{X}: f(x) \neq 0 \}|$.
We will be primarily interested in the case that $\mathcal{X} = \mathbb{Z}$, the integers.  The subset of the integers $\{a, a+1, \dots, b-1, b\}$ will be denoted by $\llbracket a, b \rrbracket$.  When $a=0$ we will abbreviate $\llbracket 0 , b \rrbracket$ by $\llbracket b \rrbracket$.

\begin{defn}[R\'enyi Entropy \cite{Ren61}] \label{defn: Renyi entropy}
For $X$ a random variable taking values $x \in \mathcal{X}$, such that $f_X(x) = \mathbb{P}(X= x)$, define for $\alpha \in (0,1) \cup (1,\infty)$, the $\alpha$-R\'enyi entropy of $X$,
\begin{align*}
    H_\alpha(X) \coloneqq (1-\alpha)^{-1} \log \sum_{x \in \mathcal{X}} f_X^\alpha(x).
\end{align*}
For $\alpha \in \{0,1,\infty \}$ the R\'enyi entropy is defined through continuous limits; 
\begin{align*}
    &H_0(X) \coloneqq \log |\{x \in \mathcal{X}: f_X(x) > 0 \}|
        \\
     &H_1(X) \coloneqq - \sum_{x \in \mathcal{X}} f_X(x) \log f_X(x)
        \\
    &H_\infty(X) \coloneqq - \log \|f_X\|_\infty.
\end{align*}

\end{defn} 
Note that $H_1(X)$ agrees with the usual Shannon entropy. As such we will employ the conventional notation $H(X) \coloneqq H_1(X)$.  
  Note that for $\alpha \in (0,1)\cup (1,\infty)$ and $\alpha' = \alpha/(\alpha - 1)$ we have the expression $H_\alpha(X) = - \alpha' \log \|f\|_\alpha$.  We will also use the notation $H_\alpha(f_X)$ in place of $H_\alpha(X)$ when it is more convenient to express the entropy as a function of the densities rather than variables.  We take $\log$ as the natural logarithm.

The R\'enyi entropy has a well known analog in the continuous setting, when $X$ is a random variable {taking values in} $\mathbb{R}^d$ { and its distribution has} a density function with respect to the usual $d$-dimensional Lebesgue measure. {The  (differential or continuous)} R\'enyi entropy is defined as
\begin{align*}
    h_\alpha(X) \coloneqq (1-\alpha)^{-1} \log \int_{\mathbb{R}^d} f_X^\alpha(x) dx
\end{align*}
for $\alpha \in (0,1) \cup (1,\infty)$. { It is extended through continuous limits\footnote{Explicitly, 
    $h_0(X) \coloneqq \log |\{x \in \mathbb{R}^d: f_X(x) > 0 \}|_d
    $, where $| \cdot |_d$ denotes the Lebesgue volume, $h_1(X) = h(X) \coloneqq - \int_{\mathbb{R}^d} f_X(x) \log f_X(x) dx$ corresponding to the differential Shannon entropy. and $h_\infty(X) \coloneqq - \log \|f_X\|_\infty$.  Where $\|f_X\|_\infty$ denotes the essential suprema of $f_X$ with respect to the Lebesgue measure in this case.} to $\alpha \in \{0,1,\infty \}$.

{Superadditivity} properties of the R\'enyi entropy connect anti-concentration results in Probability \cite{RV15, BC15:1, MMX17:2}, the seminal entropy power inequality due to Shannon {and Stam} \cite{Sha48, Sta59}, and the Brunn-Minkowski inequality of convex geometry, see \cite{Gar02} for further background.  Such connections can be traced back to \cite{CC84} where the analogy between the Brunn-Minkowski inequality and the Entropy Power Inequality was first described.  In \cite{DCT91}, proofs of the Brunn-Minkowski inequality \cite{BL76a} and entropy {power inequality} \cite{Lie78} {based on the sharp Young inequality} (see \cite{Bec75}) were synthesized to prove R\'enyi entropy inequalities connecting the two results; {an alternate unification using rearrangements was given in} \cite{WM14}.  We direct the reader to \cite{MMX17:1} for further background. There has been significant recent interest and progress in understanding the behavior of the differential R\'enyi {entropies} on independent summation.   

In analogy with the Shannon entropy power { $N(X) \coloneqq N_1(X) \coloneqq e^{2 h_1(X)/d}$}, define for $\alpha \in [0,\infty]$, $N_\alpha(X) = e^{2 h_\alpha(X)/d}$.  

\begin{thm}[\cite{BC14, BC15:1, RS16, MM19}]
For $\alpha \in [1,\infty]$, there exists $c(\alpha) \geq 1/e$ such that, for independent $\mathbb{R}^d$-valued random variables $X_i$, 
    \begin{align} \label{eq: continuous renyi epi}
        N_\alpha(X_1 + \cdots + X_n) \geq c(\alpha) \sum_{i=1}^n N_\alpha(X_i).
    \end{align}
Further, for $\alpha \in [0,1)$, there exists $c(\alpha) >0$, such that if $X_i$ are further assumed to be log-concave\footnote{We recall that an $\mathbb{R}^d$-valued random variable is log-concave when it has a density $f$ such that $t \in [0,1]$ and $x,y \in \mathbb{R}^d$ implies $f((1-t)x + ty) \geq f^{1-t}(x) f^t(y)$.}, then \eqref{eq: continuous renyi epi} holds.
\end{thm}

Note that $c(\alpha)$ can be given independent of $n$.  When $\alpha = 1$ the Entropy Power Inequality is the fact that one may take $c(\alpha) = 1$, while the Brunn-Minkowski inequality\footnote{{The Brunn-Minkowski inequality states that $|A + B|^{\frac 1 d} \geq |A|^{\frac 1 d} + |B|^{\frac 1 d}$ for non-empty Borel measurable $A, B \subseteq \mathbb{R}^d$, with addition of sets given by Minkowski addition.  For independent Borel random variables $X$ and $Y$, the support of $X + Y$, $supp(X+Y)$, is exactly the Minkowksi sum of the supports, $supp(X) + supp(Y)$.  Thus, applying this observation and Brunn-Minkowski, $N_0^{\frac 1 2}(X+Y) = |supp(X + Y)|^{\frac 1 d} \geq |supp(X)|^{\frac 1 d} + |supp(Y)|^{\frac 1 d} = N_0^{\frac 1 2}(X) + N_0^{\frac 1 2}(Y)$.  One obtains \eqref{eq: Renyi BMI} by induction.} Observe that \eqref{eq: Renyi BMI} is stronger than the inequality $N_0(X_1 + \cdots + X_n) \geq \sum_{i=1}^n N_0(X_i)$.} can be written as
\begin{align} \label{eq: Renyi BMI}
    N_0^{\frac 1 2} (X_1 + \cdots + X_n) \geq \sum_{i=1}^n N_0^{\frac 1 2} (X_i).
\end{align}
For all other $\alpha$, necessarily $c(\alpha) < 1$.  In fact without concavity assumptions on the $X_i$, \eqref{eq: continuous renyi epi} fails (see \cite{LMM20}) for any $\alpha \in (0,1)$.  Variants of the R\'enyi entropy power inequalities were studied in \cite{BM17, Li18:1} and connections with optimal transport theory can be found in \cite{Rio18,CFP18}.  %This work can be viewed as a piece of a more general program bridging the topics of information theory with convex geometry and more general geometric functional analysis.

In the discrete setting, it is natural to wonder if a parallel interplay exists, especially in light of the fruitful analogy between additive combinatorics and convex geometry, already well known, see \cite{TV06:book}.  There has been considerable interest in developing discrete versions of the entropy power inequality, see \cite{HV03, HAT14,JY10,MWW19, WM15:isit}.  General superadditivity properties of the R\'enyi entropy on independent summation have proved elusive in the discrete setting, and the mentioned results only succeed in the $\alpha = 1$ case, for special classes of variables. Even in the $\alpha = 2$ case, sometimes referred to as the collision entropy in the literature, which has taken a central role in information theoretic learning applications, see \cite{Pri10:book}, little seems to be known.

\begin{defn}
For $X$ a discrete random variable on $\mathbb{Z}$, and $\alpha \in [0,\infty]$ set
\begin{align*}
    \Delta_\alpha(X) \coloneqq e^{2 H_\alpha(X)} - 1.
\end{align*}
% Further (through continuous limits), {
% \begin{align*}
% &\Delta_0(X)  \coloneqq  |\{n : p_n > 0 \}|^2 - 1 \\ 
% &\Delta_1(X)  \coloneqq  \prod_{n \in \mathbb{Z}} p_n^{-2p_n} -1 \\ &\Delta_\infty(X)  \coloneqq   \| p_i \|_\infty^{-2} - 1.    
% \end{align*}
% }
\end{defn}
% For example, when $H_\alpha$ is taken in base $2$, 
% $\Delta_\alpha(X)=2^{2 H_\alpha(X)}-1 \in [0,\infty]$ for any $\alpha \in [0,\infty]$.  
{ Note that just as the Brunn-Minkowski inequality was written in \eqref{eq: Renyi BMI} as the superadditivity of the functional $N_0^{1/2}$ (which is stronger than that of $N_0$), the Cauchy-Davenport theorem (on $\mathbb{Z}$) can be written as the superadditivity\footnote{In fact, $N_0^{1/2}$ and $\Delta_{CD}$ turn out to satisfy the stronger property of fractional superadditivity on $\mathbb{R}$ and $\mZ$ respectively, as recently shown by \cite{BM21}.} of the functional  
$\Delta_{CD}(X)=(\Delta_0(X) + 1)^{\frac 1 2} -1$ (which is stronger than that of $\Delta_0$). }

{
\begin{defn}[Bernoulli Sum]
The class $\mathcal{B}_n$ of Bernoulli $n$-sums is defined to be the set of distributions of all random variables $Y_n = X_1 + \cdots + X_n$, where $X_i$ are independent Bernoulli random variables (i.e., $\mathbb{P}(X_i = 1) = p_i = 1 - \mathbb{P}(X_i = 0)$
for $p_i\in [0,1]$).

The class $\mathcal{B}$ of finite Bernoulli sums is given by
$$
\mathcal{B}=\bigcup_{n\in\mathbb{N}} \mathcal{B}_n ,
$$
and the class $\bar{\mathcal{B}}$ of Bernoulli sums is the closure of $\mathcal{B}$ in the weak-* topology (convergence in distribution). 
\end{defn}

As is common, we identify random variables with their distributions for ease of discussion: thus $Y$ is a  Bernoulli sum (we abuse notation to write $Y\in \mathcal{B}$ instead of saying that the law of $Y$ lies in $\mathcal{B}$) if there exists a sequence of finite Bernoulli sums $Y_n$ converging in distribution to $Y$. 
 We will show that if $S$ is a Bernoulli sum}, and $\alpha \in [2,\infty]$, then
\begin{align*}
    2\alpha' \Var(S) \leq \Delta_\alpha(S).
\end{align*}
This inequality is sharp for Bernoulli random variables with parameter $p$ tending to $0$, and for Poisson random variables with parameter $\lambda$ tending to $0$ as well.  In fact for Bernoulli sums, the functional $\Delta_\alpha(X)$ is, up to absolute constants, equal to the  variance when $\alpha \geq 2$ .  This should be compared to the continuous setting where it is known that the entropy power is proportional to variance for Gaussian {random variables}.
%In the $\alpha = 1$ case, there has been considerable effort to understand to what extent more general classes, like log-concave measures, share this property. In fact, whether a general log-concave variable's entropy is proportional to its co-variance, is in fact equivalent to, 
{ The connection between variance and entropy power has since been extended to more general log-concave random variables and R\'enyi entropies, and this connection has proved quite useful (see, e.g., \cite{BM11:cras, BM11:it, BM12:jfa} for a connection to  Bourgain's hyperplane conjecture \cite{Bou86}, or \cite{MNT21} for connections to capacities of communication channels)}.
% We recall a recent inequality \cite{BMM20}, that
% \begin{align*}
%      H_\alpha(X) \leq \frac 1 2 \log \left( 1 + \frac{4(3\alpha -1)}{\alpha-1}  \Var(X) \right)
% \end{align*}
% holds for all integer-valued random variables $X$ and $\alpha >1$.  Note that $\frac{4(3\alpha -1)}{\alpha-1} = 12 \alpha' - 4/\alpha \leq 12 \alpha'$.  Thus in our notation this gives a general upper bound for the R\'enyi entropy power by a constant multiple of the variance, for $S$ in the closure of the space of all Poisson-binomial distributions and $\alpha \geq 2$
% \begin{align*}
%     2\alpha'  \Var(S)  \leq \Delta_\alpha(S) \leq 12 \alpha'  \Var(S).
% \end{align*}
% Let us remark that the bound $2 \alpha'  \Var(S) \leq \Delta_\alpha(S)$ 

Further we will prove a ``min-entropy power inequality'': for $\alpha = \infty$, we will prove that, without qualification, independent { random variables} $X_i$ satisfy the following R\'enyi entropy inequality:
\begin{align*}
    \Delta_\infty(X_1 + \cdots + X_n) \geq c \sum_{i=1}^n \Delta_\infty(X_i),
\end{align*}
for a universal $c >0$ independent of $n$, the number of summands. As will be shown, one can take $c = \frac 1 {20}$ and we will observe below that necessarily $c \leq \frac 1 2$.

As an application, we {develop new bounds for} a generalized version of the Littlewood-Offord problem.  {Furthermore}, we give sharp bounds for a R\'enyi entropic generalization of the Littlewood-Offord problem {with} $\alpha \geq 2$.  %Thus we arrive at a R\'enyi entropic generalization of the Littlewood-Offord bounds.

The mathematical underpinning of the min-EPI is an identification of extreme points in the space of probability measures with a fixed upper bound on their density functions that was proven in \cite{MMX17:2}, and a rearrangement inequality from \cite{MWW19}.
A main technical contribution is an $L^p$-norm bound on the characteristic function of a Bernoulli random variable.
Recall that $\hat{f}_{X}(t)=\mathbbm{E}(e^{itX})$, $t \in \mathbb{R}$, denotes the {Fourier} transform of the (discrete) random variable $X$. %As for norm is concerned,
For $q \geq 1$ we set  
$\| \hat{f}_{X}\|_q^q := \frac{1}{2 \pi} \int_{-\pi}^\pi |\mathbb{E}e^{itX}|^q dt$.  We prove that when $X$ is a Bernoulli, with variance $\sigma^2$,
$\| \hat{f}_X \|_q^q \leq ( 6 \sigma^2 q )^{-1/2} \int_0^{\sqrt{6 \sigma^2 q}} e^{-t^2/2} dt$,
{ the constant $6$ being optimal.}

Let us outline the contents of the paper.  In Section \ref{sec: Bernoulli Sums} we derive $L^p$ bounds for the characteristic function of a Bernoulli random variable in terms of its variance using a distributional argument% due to \cite{NP00}
.  Then, we give a general theorem for the extension of such inequalities to independent sums.  In Section \ref{sec: Renyi entropy inequalities} we demonstrate how the characteristic function bounds of Section \ref{sec: Bernoulli Sums} can be used to deliver sharp comparisons between the R\'enyi entropy and the variance for variables with distributions in the closure of the Bernoulli sums.   It is also demonstrated that such bounds cannot be achieved for the case that $\alpha = 1$ by a counter example. In Section \ref{sec: min-EPI} we develop the functional analytic tools to reduce the problem of a min-EPI for general random variables to a min-EPI of variables consisting only of Bernoulli and uniform distributions.  In addition, %an extension of the inequality to $\mathbb{Z}^n$ is given, as well as
reversals and sharpenings of the min-EPI in the case that the $X_i$ are Bernoulli sums.  In Section \ref{sec: littlewood-offord problem}, the Littlewood-Offord problem is introduced, and its reduction to Bernoulli sums is given.  The bounds for the min-Entropy of a Bernoulli sum in terms of its variance are applied, and then it is shown that these results can be extended to deliver R\'enyi bounds on an entropic Littlewood-Offord problem.  Some proofs are suppressed to the appendix.

% \iffalse
% \begin{thm} \label{thm: rearrangement of extremes}
% For $X_1, \dots, X_n, Y_1, \dots, Y_k$ independent integer-valued random variables taking values in a finite subset of the integers such that $H_\infty(X_i) \geq \log n_i$, for integer $n_i \geq 2$, and $H_\infty(Y_j) = \log \varepsilon_j$ for $\varepsilon_j \in (1,2)$ then exists $Z_1, \dots, Z_n, B_1, \dots, B_k$ independent such that $Z_i$ are uniform on $\{0,1, \dots, n_i-1\}$ and $B_j$ are Bernoulli with $H_\infty(B_j) = H_\infty(Y_j)$
% \[
%     H_\infty(X_1 + \cdots + X_n + Y_1 + \cdots + Y_k) \geq H_\infty(Z_1 + \cdots + Z_n + B_1 + \cdots + B_k).
% \]
% \end{thm}
% \begin{proof}
% This should leverage the 
% \end{proof}
% Using the notation $\Delta_\infty(X) = e^{2H_\infty(X)} -1$
% \begin{cor}
% There exists a constant $c >0$ such that for any $X_i$ independent integer-valued random variables, there exists $Z_i$ independent integer-valued random variables, such that $Z_i$ is either uniform on a set $\{0, 1, \dots, n \}$ or Bernoulli and
% \[
%     \Delta_\infty(X_1 + \cdots + X_n) \geq \Delta_\infty(Z_1 + \cdots + Z_n)
% \]
% while
% \[
%     \sum_i \Delta_\infty(Z_i) \geq c \sum_i \Delta_\infty(X_i).
% \]
% \end{cor}
% \subsection{Hypercube}
% \fi

\section{Bernoulli Sums} \label{sec: Bernoulli Sums}
 For fixed $m$, $\mathcal{B}_m$ sums
%(which must be independent, but not necessarily identically distributed, summands)
corresponds to the normalized Polya frequency sequences of length $m$, whose probabilistic behavior was studied by Pitman \cite{pitman1997probabilistic}. The class $\bigcup_{n=1}^m \mathcal{B}_n$ is a subset of the ultra log-concave distributions of order $m$, studied by Pemantle \cite{pemantle2000towards} (see also \cite{liggett1997ultra}) and written ULC$(m)$.  The ULC$(m)$ variables can be understood as all variables with distribution log-concave\footnote{A sequence $x_n$ is log-concave, when it has contiguous support and $x_n^2 \geq x_{n+1} x_{n-1}$.  When $x_n$ is the distribution of a random variable $X$, that is $x_n = \mathbb{P}(X = n)$, we say that the distribution of $X$ is log-concave with respect to the distribution of a variable $Y$, if $y_n \coloneqq \mathbb{P}(Y=n)$ admits a log-concave sequence $a_n$ such that $x_n = a_n y_n$.  When $Y$ is binomial$(p,m)$ we note that the statement is independent of $p \in (0,1)$.}  with respect to a binomial distribution.  See  \cite{branden2020lorentzian,anari2018log,chan2021log} for important connections with combinatorics. 

 Bernoulli sums of unbounded length have distributions log-concave with respect to a Poisson random variable.  This class often simply called  ``ultra log-concave'', and arise naturally in the study of intrinsic volumes (see \cite{marsiglietti2022Concentration, schneider2014convex}) and in connection with entropic limit theorems. For example, Harrem\"oes \cite{Har01} showed that the Poisson distribution has maximum entropy among all distributions in $\bar{\mathcal{B}}$ with fixed mean; see also \cite{Joh06, Yu09:2, HJK10, JKM13, MM22} for other related work.  See \cite{HJ17, HJ19, TT19, marsiglietti2020geometric,melbourne2021discrete, saumard2014log} for recent results and further background on Bernoulli sums as well as log-concave probability sequences.

%and each $Y_n$ is the sum of independent Bernoulli random variables $X_i(n)$, that is $Y_n \xRightarrow{w} Y$ and $Y_n = \sum_{i=1}^{m(n)} X_i(n)$ where $X_i(n)$ are all independent.

We will first derive $L^p$ bounds on the characteristic functions of Bernoulli random variables via a comparison with Gaussians.  This argument will be distributional.  We use $V$, the capitalization of a non-negative function $v$ to denote its distribution function.  Explicitly,
\begin{defn}
    For measurable $v: \mathbb{R} \to [0,\infty)$, its distribution function $V:[0,\infty]$, is defined by
    \[
        V(t) = | \{ x: v(x) > t \} |.
    \]
\end{defn}
Integrals of functions can be easily computed from integrals of their distribution functions using Fubini-Tonelli, and the following formula
\begin{equation} \label{eq: Distribution function integral formula}
    \int v = \int_0^\infty | \{ x : v(x) > \lambda \} | d\lambda = \int_0^\infty V(\lambda) d \lambda.
\end{equation}
{ If we denote $[x]_+ \coloneqq \max \{ x , 0 \}$ then Fubini-Tonelli admits a slight generalization of \eqref{eq: Distribution function integral formula}, for $t \geq 0$
\begin{equation} \label{eq: generalized distribution function integral formula}
    \int [ v(x) - t]_+ dx  = \int_0^\infty | \{ x : v(x) > t + \lambda \} | d \lambda = \int_t^\infty V(\lambda) d\lambda.
\end{equation}

\begin{lem} \label{lem: Nazarov Tricked}
    For $w$ and $v$ non-negative functions in $L^1$, such that $\int w \geq \int v$, with distribution functions $W$ and $V$ respectively, then, if $W-V \leq 0$ on $[0,t_0]$ and $W - V \geq 0$ on $[t_0, \infty)$, it holds
        \[
           \int w^p \geq \int v^p
        \]
        for $p \geq 1$.
\end{lem}

\begin{proof}
For $\varphi(x)$ convex and smooth, and $x \geq 0$, Taylor expansion gives,
\begin{align*}
    \varphi(x) = \varphi(0) + x \varphi'(0) + \int_0^\infty [x - t]_+ \varphi''(t) dt.
\end{align*}
For $\varphi(0) = \varphi'(0) = 0$, applying the Taylor expansion and then \eqref{eq: generalized distribution function integral formula},
\begin{align*}
    \int \varphi(w(y)) dy
=
            \int_0^\infty \left(\int [w(y) - t]_+ dy \right) \varphi''(t) dt
=
            \int_0^\infty \left(\int_t^\infty  W(\lambda) d\lambda \right) \varphi''(t) dt .
\end{align*}
Thus to prove $\int \varphi(w) \geq \int \varphi(v)$, it suffices to prove 
$
    \Psi(t) = \int_t^\infty W(\lambda) - V(\lambda) d\lambda \geq 0,
$
for $t \geq 0$.
Since $\Psi(0) = \int w - \int v \geq 0$, by assumption, and $\Psi'(t) = V(t)- W(t) \geq 0$ on $[0,t_0]$ so that $\Psi(t) \geq 0$ on $[0,t_0]$.  For $t \geq t_0$ the result is immediate from $W(\lambda) - V(\lambda) \geq 0$ for $\lambda \geq t \geq t_0$.  Taking $\varphi(x) = x^p$ gives the result.
\end{proof}

The connection between ``hockey stick'' integrands and general convex functions is well known in the theory of majorization \cite{Cho74,Joe87}.  The result above can also be obtained through the lemma of Nazarov-Podkorytov  \cite{NP00}.  See \cite{melbourne2021transport} for more on the connections between the lemma of Nazarov-Podkorytov and majorization, and see \cite{melbourne2021quantitative} where the lemma is used to derive a quantiative min-entropy power in the continuous setting.}
This lemma will be used to derive the following main theorem.
\begin{thm} \label{thm: Bernoulli Lp Bounds}
    For $X$ a Bernoulli with variance $\sigma^2$, then $q \geq 1$ implies
    \begin{align*}
        \frac{1}{2 \pi} \int_{-\pi}^\pi |\mathbb{E}e^{itX}|^q dt \leq \frac 1 { \sqrt{6 \sigma^2 q}} \int_0^{ \sqrt{6 \sigma^2 q} } e^{-t^2/2} dt.
    \end{align*}
\end{thm}

{ The constant $6$ in inequality in Theorem \ref{thm: Bernoulli Lp Bounds}, is the best possible\footnote{Note it can be shown that the function $\Phi(x) = \frac{\int_0^x e^{-t^2/2} dt }{x}$ is decreasing in $x$, so that larger constants reflect a stronger inequality.}, and allows the derivation of several sharp inequalities in the sequel.  However, at a very small loss, a completely elementary argument, for which we thank an anonymous reviewer, yields the following.
\begin{prop} \label{prop: reviewer}
    For $X$ a Bernoulli with variance $\sigma^2$, then $q \geq 1$ implies
    \begin{align*}
        \frac{1}{2 \pi} \int_{-\pi}^\pi |\mathbb{E}e^{itX}|^q dt \leq \frac 1 { \sqrt{4 \sigma^2 q}} \int_0^{ \sqrt{4 \sigma^2 q} } e^{-t^2/2} dt.
    \end{align*}
\end{prop}

\begin{proof}
Observe that Bernoulli $X$ with variance $\sigma^2$, the norm squared of its characteristic function can be expressed as a convex combination of the cosine function and the constant function one, namely 
\begin{align}\label{eq: cosine expression of char. func}
    |\mathbb{E} e^{itX}| = \sqrt{(1-\lambda) + \lambda \cos (t)},
\end{align} where $\lambda = 2 \sigma^2 \in [0, 1/2]$.
Using $1-x \leq  e^{-x}$
 together with $\sin s \geq  \frac 2 \pi s$ for $0 \leq s \leq \pi / 2$, we have
 \begin{align*}
     1- \lambda + \lambda \cos t \leq e^{- \lambda (1-\cos t)} = e^{- 2\lambda \sin^2(t/2)}  \leq e^{-2 \lambda t^2/\pi^2}
 \end{align*}

Hence with $\lambda = 2 \sigma$
\begin{align*}
    \frac{1}{2 \pi} \int_{-\pi}{\pi} |\mathbb{E} e^{it X} |^q dt 
        &=
            \frac 1 {2 \pi} \int_{- \pi}^{\pi} | 1 - \lambda + \lambda \cos t |^{q/2} dt
                \\
        &\leq 
           \frac 1 {2 \pi} \int_{- \pi}^{\pi} e^{-\lambda qt^2/\pi^2} dt 
                \\
        &=
            \frac{1}{\sqrt{2 \lambda q}} \int_0^{\sqrt{2 \lambda q }} e^{-t^2/2} dt
                =
                 \frac{1}{\sqrt{4 \sigma^2 q}} \int_0^{\sqrt{4 \sigma^2 q }} e^{-t^2/2} dt.   
\end{align*}
\end{proof}

To pursue the sharp inequality, the expression \eqref{eq: cosine expression of char. func} and Lemma \ref{lem: Nazarov Tricked} motivate }the following definitions for $\lambda \in [0,1/2]$
\begin{align*}
    v_\lambda(t) \coloneqq \sqrt{(1-\lambda) +\lambda \cos t }
        \hspace{8mm}
    w_\lambda(t) \coloneqq \exp \left[\frac{-3 \lambda t^2}{ 2\pi^2} \right].
\end{align*}
We first claim that $\int_0^\pi v_\lambda^s(t) dt < \int_0^\pi w_\lambda^s(t) dt$ holds when $s = 1$.

\begin{lem}\label{lem: base case}
When $s = 1$, and $\lambda \in (0,1/2]$, $\int_0^\pi v_\lambda^s(t) dt > \int_0^\pi w_\lambda^s(t) dt$, or
\begin{align*}
    \int_0^\pi \sqrt{(1-\lambda) + \lambda \cos(t)}  dt
        < \int_0^\pi  \exp\left\{\frac{-3 \lambda t^2}{2 \pi^2} \right\} dt.
\end{align*}
\end{lem}

\begin{proof}
Using the inequality $\sqrt{ 1- x} < 1- \frac x 2$ for $x \in (0,1]$,
\begin{align*}
    \int_0^\pi \sqrt{(1-\lambda) + \lambda \cos(t)}  dt
        <
            \int_0^\pi 1 - \lambda \frac{ 1- \cos(t)}{2} dt
        =
            \pi (1 - \lambda/2).
\end{align*}
Meanwhile the inequality and $e^x \geq 1 + x$, gives
\begin{align*}
    \int_0^\pi \exp \left[ \frac{-3 \lambda t^2}{2 \pi^2} \right] dt
        \geq
            \int_0^\pi 1 + \frac{-3 \lambda t^2}{2 \pi^2} dt
        =
            \pi(1 -  \lambda/2),
\end{align*}
so that $\int_0^\pi w_\lambda(t) dt > \int_0^\pi v_\lambda(t) dt$ as claimed.
\end{proof}

% As functions of $\lambda$ the left hand side is affine $\pi - \pi \lambda$, while the right hand is convex,
% \begin{align}
%     \frac{d}{d \lambda}\bigg|_{\lambda = 0} \int_0^\pi e^{-3\lambda t^2/\pi^2} dt
%         = - \frac{3}{\pi^2} \left[\int_0^\pi t^2 e^{-3 \lambda t^2/\pi^2} dt \right]_{\lambda = 0} = -\frac{3}{\pi^2} \left[ \frac{t^3}{3} \right]_{0}^\pi = - \pi.
% \end{align}
% So the right hand side has the left hand side as its first order approximation and the claim follows.

\begin{lem} \label{lem: no more than one zero}
For $\lambda >0$, the function $t \mapsto w_\lambda(t) - v_\lambda(t)$ has no more than one zero on $(0,\pi]$
\end{lem}

The proof is calculus computations and we leave it to an appendix.

\begin{proof}[Proof of Theorem \ref{thm: Bernoulli Lp Bounds}]
As functions of $t$, $v_\lambda$ and $w_\lambda$ are both strictly decreasing on $[0,\pi]$.  Thus, their respective distribution functions $V_\lambda$ and $W_\lambda$ are just their strictly decreasing inverse functions on $[0,1]$.  Since $w_\lambda > v_\lambda$ for small $t$, $W_\lambda > V_\lambda$ for $y$ close to $1$, and since $w_\lambda$ and $v_\lambda$ cross at no more than one point (by Lemma \ref{lem: no more than one zero}), $V_\lambda$ and $W_\lambda$ cross at no more than one point.  We consider two cases. If  $w_\lambda(\pi) - v_\lambda(\pi) \geq 0$, then $w_\lambda - v_\lambda \geq 0$ on $[0,\pi]$ and the theorem holds immediately.  If $w_\lambda(\pi) - v_\lambda (\pi) <0$, then $w_\lambda - v_\lambda$ has exactly one zero, and hence $W_\lambda- V_\lambda$ has exactly one zero. {This in concert with Lemma \ref{lem: base case} shows that $W_\lambda$ and $V_\lambda$  satisfy the conditions of Lemma \ref{lem: Nazarov Tricked}. Thus 
\begin{align*} 
\int_0^\pi w_\lambda^s(t) dt > \int^\pi  v_\lambda^s(t) dt
\end{align*}
for all $s \geq 1$. This is equivalent to our conclusion since,}
$$
\frac{1}{2 \pi} \int_{-\pi}^\pi |\mathbb{E}e^{itX}|^q dt = \frac{1}{\pi} \int_0^\pi v_\lambda^q, 
$$
and 
$$
\frac{1}{\pi} \int_0^\pi w_\lambda^q
=
\frac{1}{\pi} 
\int_0^\pi e^{-\frac{3q\lambda t^2}{2 \pi^2}}dt 
=
\frac 1 { \sqrt{3 q \lambda}} \int_0^{ \sqrt{3q \lambda} } e^{-u^2/2} du 
= 
\frac 1 { \sqrt{6 \sigma^2 q}} \int_0^{ \sqrt{6 \sigma^2 q} } e^{-u^2/2} du 
$$
where the second inequality follows from the change of variable 
$u = \sqrt{3q\lambda}t/\pi$ and the last one
from the fact that $\lambda=2\sigma^2$.
\end{proof}

Our next aim is to extend the previous comparison to a finite or infinite sum of independent Bernoulli random variables. We will use the following lemma.

\begin{lem}\label{lem: Fourier bound to fourier bound on sum}
Fix $\Phi \colon (0,\infty) \to [0,\infty)$.
Suppose that $X_i$ are independent random variables such that
\begin{align*}
    \| \hat{f}_{X_i} \|_q^q \leq \Phi(c_i q)
\end{align*}
holds for all $q \geq 1$ and some $c_i >0$. Then
\begin{align*}
    \| \hat{f}_{\sum_i X_i} \|_q^q \leq \Phi( c q)
\end{align*}
holds for all $q \geq 1$ as well, with $c = \sum_i c_i$.
\end{lem}

\begin{proof}
    By independence and H\"older's inequality for $\sum_i \frac 1 {q_i} = 1$,
    \begin{align*}
        \| \hat{f}_{\sum_i X_i}\|_q^q
            =
                \| \prod_i \hat{f}_{X_i}^q \|_1
            \leq
                \prod_i \| \hat{f}_{X_i}^q \|_{q_i}
            =
                \prod_i \left( \| \hat{f}_{X_i} \|_{q q_i}^{q q_i} \right)^{\frac 1 {q_i}}.
    \end{align*}
Applying the hypothesis, and taking $q_i = \frac{c}{c_i}$,
\begin{align*}
    \| \hat{f}_{\sum_i X_i} \|_q^q
        \leq
            \prod_i \Phi^{\frac 1 {q_i}} ( c_i q_i q)
        =
          \Phi( cq).
\end{align*}
\end{proof}

Thanks to the previous Lemma, the bound of Theorem \ref{thm: Bernoulli Lp Bounds} transfers to Bernoulli sums.

\begin{thm} \label{thm: Bernoulli sum lp fourier bound}
Let $Y$ be a Bernoulli sum with variance $\sigma^2$.
%, in the sense that $Y = \sum_{i=1}^n X_i$  for $X_i$ independent Bernoulli random variables of variance $\sigma_i^2$ such that $\sum_{i=1}^n \sigma^2 = \sigma^2$.  
Then
\begin{align} \label{eq: Bernoulli sums Lp fourier bound}
    \| \hat{f}_Y \|_q^q \leq \frac{1}{\sqrt{6 \sigma^2 q }} \int_0^{\sqrt{6 \sigma^2 q}} e^{-t^2/2} dt.
\end{align}
\end{thm}

\begin{proof}
Since $Y$ is a Bernoulli sum, there exist a sequence of $Y_n$ converging weakly to $Y$.  Given a $Y_n$ of the sequence, there exists $X_i$ independent Bernoulli with variance $\sigma_i^2$, such that $\sum_{i=1}^{m(n)} X_i$. Taking $c_i =\sigma_i^2$, $\Phi(x) = \frac{1}{\sqrt{6 x}} \int_0^{\sqrt{6x}} e^{-t^2/2} dt$, the hypothesis of Lemma \ref{lem: Fourier bound to fourier bound on sum} is satisfied for $X_i$ thanks to Theorem \ref{thm: Bernoulli Lp Bounds}, and the conclusion of the Lemma is
exactly \eqref{eq: Bernoulli sums Lp fourier bound}. 

Since the bound holds for each $Y_n$, it is enough to observe $\| \hat{f}_{Y_n} \|_q^q \to \| \hat{f}_{Y} \|_q^q$ and $\sigma^2_{Y_n} \to \sigma^2_{Y} < \infty$.  Note, that $Y$ is necessarily log-concave since is the limit of log-concave variables $Y_n$ (see Definition \ref{def:log-concave} for a definition). As a consequence there exists $C, c >0$ such that $f_{Y_n}(k) \leq Ce^{-c |k|}$ and $f_Y(k) \leq Ce^{-c|k|}$ holds for all $n$, and all $k \in \mathbb{Z}$.  Thus all moments exist and there is no difficulty passing limits, and by L\'evy's continuity theorem $\| \hat{f}_{Y_n} \|_q^q \to \| \hat{f}_{Y} \|_q^q$ to complete the result.
\end{proof}

Let us remark on the nature of the function $z \mapsto \frac{1}{z}\int_0^z e^{-t^2/2} dt$, as it will be useful to have simple upper bounds for our applications.

\begin{lem} \label{lem: Gaussian integral bound}
For $z \in (0,\infty)$,
    \begin{align*}%\label{eq: The gaussian bound near zero}
        \frac{1}{z}\int_0^z e^{-t^2/2} dt \leq \min \left\{ \frac 1 {\sqrt{ 1 + (z^2/3)}} , \sqrt{\frac \pi { 2z^2}} \right\}.
    \end{align*}
\end{lem}

{The proof is computational and included in the appendix \cite{BernoulliAppendix}.
\begin{remark}
Note that 
$
    \sqrt{\frac{\pi}{2 z^2}} \leq \left( 1 + \frac{z^2}{3} \right)^{-1/2}
$
exactly when $z \geq \sqrt{\frac{3 \pi}{6-\pi}} \approx 1.8158$.
\end{remark}

%%%%%%%%%%%%%%%%%%%%%%%%%%%%%%%%%%%%%%%%
%%%%%%%%%%%%%%%%%%%%%%%%%%%%%%%%%%%%%%%%%%%%%%%%%%%

\section{R\'enyi Entropy Inequalities} \label{sec: Renyi entropy inequalities}

In this section we prove R\'enyi entropy inequalities for Bernoulli sums and their limits. We will use a less orthodox formulation of the well-known Hausdorff-Young's inequality to translate $L^{\alpha'}$ bounds on the characteristic functions of Bernoulli sums into $\alpha$-R\'enyi entropy bounds.

\begin{thm}[Hausdorff-Young]
For $p \in [2,\infty]$, and a random variable $X$ on $\mathbb{Z}$ with probability mass function $f$ then $\| f \|_p \leq \| \hat{f} \|_q$ where $\frac 1 p + \frac 1 q = 1$.
\end{thm}

%{Cyril: Hausdorff-Young is usually stated the opposite: namely
%$\| \hat{f} \|_p \leq \| f \|_q$, with $p \in [1,2]$. I wonder if there is a proper %terminology here. 

%Also, $\| f \|_p \leq \| \hat{f} \|_q$ for $p \in [1,2]$ is most certainly false. I wonder anyway if, in our context of only Bernoullis and Poisson-binomial, anything could be true  with possibly additional constants ? Might be a very stupid question.}

%{\color{blue} James: Agreed, I am not sure of a best practice here with Hausdorf-Young.  Finding some reversal is a bit enticing.  The formula for the bounds we derive for the Renyi entropy seem plausible for $\alpha < 2$, as the anticipate the blowup at the Shannon $\alpha = 1$ case.}

\begin{remark}
We observe that, in contrast with the continuous setting, the inequality $\| f \|_p \leq \| \hat{f} \|_q$ is sharp for random variables on $\mathbb{Z}$. To see this it is enough to consider a Dirac mass at zero ($f(0)=1$ and $f(n)=0$ for all $n \neq 0$ for which $\hat{f}\equiv 1$ so that  $\| f \|_p = \| \hat{f} \|_p = 1$ for all $p$).
\end{remark}

\begin{proof}
The inequality follows by the Riesz-Thorin interpolation Theorem since $\|f\|_2 = \|\hat{f}\|_2$ and $\|f\|_\infty \leq \|\hat{f}\|_1$.
\end{proof}

\begin{thm} \label{thm: Renyi entropy by variance}
When $Y$ is a Bernoulli sum with variance $\sigma^2$ and  $\alpha \in [2,\infty]$ then
\[
    H_\alpha (Y) \geq \log \left[\frac{\sqrt{6 \sigma^2 \alpha'}}{\int_0^{\sqrt{6 \sigma^2 \alpha' }} e^{-t^2/2} dt} \right].
\]
In particular,
\[
    H_\alpha(Y) \geq \frac{1}{2} 
    \max \left\{ \log \left( 1 + 2 \alpha' \sigma^2 \right), \log \left( \frac{12 \alpha' \sigma^2}{\pi} \right) \right\}.
\]
\end{thm}

\begin{remark}
{
Notice that $H_\alpha(Y) \geq \frac{1}{2} \log \left( 1 + 2 \alpha' \sigma^2 \right)$ reads 
\begin{equation} \label{eq:PB}
\Delta_\alpha(Y) \geq 2 \alpha' \Var(Y) .
\end{equation}
Note that if $Y_p$ is a Bernoulli random variable with parameter $p$, Taylor expansion gives 
\begin{align*}
    \Delta_\alpha(Y_p) = 2 \alpha' p + o(p).
\end{align*}
Thus the constant $2 \alpha'$ is optimal in \eqref{eq:PB}.}
\end{remark}

\begin{proof}
Recall that $H_\alpha (Y)=-\alpha'\log \| f_Y \|_\alpha=-\log \| f_Y \|_\alpha^{\alpha'}$. Therefore, the Hausdorff-Young Inequality  $\| f_{Y} \|_\alpha^{\alpha'} \leq \| \hat{f}_{Y} \|_\alpha^{\alpha'}$ together with Theorem \ref{thm: Bernoulli sum lp fourier bound} guarantee that
\begin{align*}
    H_\alpha(Y) 
        \geq 
            - \log \| \hat{f}_Y \|_{\alpha'}^{\alpha'}
        \geq
           \log  \left[\frac{\sqrt{6 \sigma^2 \alpha'}} {\int_0^{\sqrt{6 \sigma^2 \alpha'}} e^{-t^2/2} dt}\right] .
\end{align*}
The last bound follows from Lemma \ref{lem: Gaussian integral bound}
applied with $z=\sqrt{6\sigma^2 \alpha'}$.
%By Lemma \ref{lem: Gaussian integral bound} we have
%\begin{align*}
%    H_p(Y) \geq \frac 1 2 \log (1+ 2 \alpha' \sigma_Y^2)
%\end{align*}
%and 
%begin{align*}
%    H_p(Y) \geq \frac 1 2 \log \left( \frac{12\alpha'\sigma_Y^2}{\pi } \right).
%\end{align*}
\end{proof}

%\begin{defn}[Discrete R\'enyi Entropy Power] \label{def: Renyi entropy power}
%For $p \in (0,1) \cup (1,\infty)$, and $X$ a discrete variable with probabilities $t_i$, %define the $p$-th R\'enyi entropy power
%\begin{align}
 %   \Delta_p(X) = \left(\sum_i t_i^p \right)^{\frac 2 {1-p}} - 1.
%\end{align}
%Define $\Delta_0(X) = | \{ t_i \}_i |^2 - 1$ where $|\cdot|$ denotes cardinality, %$\Delta_1(X) = \prod_i t_i^{-2 t_i} - 1$, and $\Delta_\infty(X) = (\max_i t_i)^{-2} - 1$.
%\end{defn}
%Note that when $H_p(X)$ is taken in base $2$, $\Delta_p(X) = 2^{2 H_p(X)} - 1$.
Our last ingredient in the proof of Theorem \ref{thm: Bernoulli REPI} is the following lemma.

\begin{lem} \label{lem: Delta versus variance}
For $\alpha \geq 2$, {if $X$ is a Bernoulli random variable, then}
%and $X$ Bernoulli implies
\begin{align} \label{eq: Delta versus variance}
    \Delta_\alpha(X) \leq 12  \Var(X),
\end{align}
with equality when the Bernoulli parameter $\theta = \frac 1 2$.
\end{lem}

\begin{remark}
Note that if $\theta = 1/2$, $\Delta_\alpha(X) = 12  \Var(X)$ independent of $\alpha$ and cannot be improved for any R\'enyi parameter $\alpha \geq 2$.  Considering the Shannon entropy for small $\theta$ shows an analogous result fails for all $\alpha \leq 1$,  as $ \lim_{\theta \to 0} \frac{\Delta_1(X)}{ \Var(X)} = \infty$.  It can be shown for $\alpha \in (1,2)$, there exists $C(\alpha)$ such that $\Delta_\alpha(X) \leq C(\alpha)  \Var(X)$.  However by investigation about $\theta$ close to $0$, $C(\alpha)$ is necessarily larger than $12$ for $\alpha < \frac 6 5$.
\end{remark}

\begin{proof}[Proof of Lemma \ref{lem: Delta versus variance}]
By the monotonicity of R\'enyi entropy $H_\alpha(X) \leq H_2(X)$, so it suffices to prove $\Delta_2(X) \leq 12  \Var(X)$.  This is equivalent that for $t \in [0,1]$,
\begin{align*}
    (t^2 + (1-t)^2)^{-2} - 1 \leq 12 t(1-t).
\end{align*}
{ This}
%Which 
is equivalent to proving $P(t) \geq 0$, where $P$ is the polynomial,
\begin{align*}
    P(t) = (12t(1-t)+ 1)(t^2 + (1-t)^2)^2  -1 \geq 0 
\end{align*}
on $[0,1]$.  However, $P$ can be factored, to
\begin{align*}
     4(1-t) t (2t-1)^2(3t^2 - 3t+2),
\end{align*}
from which the non-negativity of $P$ on the interval is obvious.
\end{proof}

%We are now in position to prove Theorem \ref{thm: Bernoulli REPI}.

The main theorem is the following.

\begin{thm} \label{thm: Bernoulli REPI}
For $\alpha \geq 2$,  $X_i$ a Bernoulli sum, then
\begin{align}\label{eq: Bernoulli REPI}
    \Delta_\alpha(X_1 + \cdots + X_n) \geq \frac{ \alpha'}{6} \sum_{i=1}^n \Delta_\alpha(X_i),
\end{align}
independent of the number of summands. 
\end{thm}

\begin{remark}
 Note that 
 $
 c(\alpha) := \frac{\alpha}{6(\alpha-1)} = \frac {\alpha'} 6\geq \frac 1 6.
 $
 Further, inequality \eqref{eq: Bernoulli REPI} fails for the Shannon entropy, as can be seen by considering $X_i$ to be iid with parameter $\theta$.  Indeed, the sum $X_1 + \cdots + X_n$ has a {binomial} distribution, whose entropy has the well known asymptotic formula,
\begin{align*}
    H(X_1 + \cdots + X_n) = \frac 1 2 \log_2( 2 \pi e n \theta (1-\theta)) + O(1/n).
\end{align*}
Therefore,
%{ Cyril: please check the following computation. I don't understand the first step since in my opinion $\Delta_1 = e^{2H_1}-1$ and not $2^{2H_1}$! I agree with the denominator.}
\begin{align*}
    \frac{\Delta_1( \sum_i X_i)}{\sum_i \Delta_1(X_i)} 
        = 
            \frac{ 2^{2 (H(X_1 + \cdots+ X_n) + O(1/n))}-1}{ n \Delta_1(X_1) }
        =
             \frac{ 2 \pi e \theta(1-\theta) 2^{O(1/n)} - \frac 1 n}{\Delta_1(X_1)} 
        \to 0,
\end{align*}
with $\theta \to 0$.
% and $H(X_1)$ taken in base $e$ goes to zero, hence, $\Delta_1(X_1) = e^{2 H(X_1) } - 1 \approx 2 H(X_1)$.  Thus
% \begin{align*}
%     \frac{\Delta_1( \sum_i X_i)}{\sum_i \Delta_1(X_i)} 
%         &\leq
%             O(1) \frac{  \Var(X_1)}{H(X_1)} \to 0,
% \end{align*}
% with $\theta \to 0$,
This precludes a summand independent entropy power inequality in the sense of Theorem \ref{thm: Bernoulli REPI} for the Shannon entropy.
%{ Cyril: Why ?}
%{\color{blue} James: I have tried to clean the argument}
\end{remark}

\begin{proof}[Proof of Theorem \ref{thm: Bernoulli REPI}]
By Theorem \ref{thm: Renyi entropy by variance},  it holds
\begin{align*}
    \Delta_\alpha\left(\sum_i X_i\right) 
        =
            e^{2 H_\alpha(\sum_i X_i)} - 1
        \geq
            2 \alpha' \Var \left( \sum_i X_i\right).
\end{align*}
Using that additivity of the variance of independent variables and Lemma \ref{lem: Delta versus variance}, we conclude that
\begin{align*}
    2\alpha' \Var \left( \sum_i X_i \right) 
        =
          2\alpha' \sum_i  \Var(X_i)
        \geq 
            2\alpha' \sum_i \frac{\Delta_\alpha(X_i)}{12}
        =
            \frac{\alpha'}{6} \sum_i \Delta_\alpha(X_i).
\end{align*}
\end{proof}

\section{Min-Entropy Power} \label{sec: min-EPI}
Given independent integer-valued random variables $X_1, \dots, X_n$.  We investigate minimizers of the quantity
$
    H_\infty(X_1 + \cdots + X_n)
$
on the set $H_\infty(X_j) \geq \log C_j$ where $C_j > 1$.  We note that $H_\infty(X) \geq \log C$ corresponds to $\|f\|_\infty \leq \frac 1 C$.

    \subsection{Extreme points}
Let us denote the set of probability density functions supported on a finite set $M$, with density $f$ bounded by $\frac 1 C$ by
\begin{align} \label{eq: set of probability measures}
    \mathcal{P}_{C}(M) = \left\{ f\colon M \to [0,1] \mbox{ such that } 
    0 \leq f \leq \frac 1 C \mbox{ and } \sum_{i \in M} f(i) = 1 \right\} .
\end{align}
Note that $\mathcal{P}_{C}(M)$ is a convex compact subset\footnote{We assume $C \leq M$ else $\mathcal{P}_{C}(M)=\emptyset$.} of $\mathbb{R}^{|M|}$, and hence is the closure of the convex hull of its extreme points.   Let us recall the necessary definitions.  The extreme points $\mathcal{E}(K)$ associated to a convex set $K$ are defined as
\[
    \mathcal{E}(K) = \left\{ k \in K : k = \frac{k_1 + k_2}{2} \hbox{ for } k_i \in K \hbox{ implies } k_1 = k_2 \right\}.
\]
The convex hull of a set $K$ is
\[
    co(K) = \left\{ x: \exists \lambda_i >0 \hbox{ and } k_i \in K \hbox{ such that } \sum_{i=1}^n \lambda_i =1, \sum_{i=1}^n \lambda_i k_i = x \right\} .
\]
For $x \in \mathbb{R}$, we write $\floor*{x} = \max \{ n \in \mathbb{Z}: n \leq x \}$ for the entire part of $x$.  

\begin{thm}
For $\mathcal{P}_{C}(M)$ defined as in \eqref{eq: set of probability measures} with $C \leq |M|$,
\[
    \mathcal{E}( \mathcal{P}_{C}(M) ) = \left\{ f: f = \frac{\mathbbm{1}_A} C + \left(1-\frac{\floor{C}}{C} \right) \mathbbm{1}_{\{x\}}, |A| = \floor{C}, x \notin A \right\}.
\]
\end{thm}
Note that when $C$ is chosen to be {a natural} number, $\mathcal{E}( \mathcal{P}_{C}(M) )$ is the uniform distributions on sets of size $C$ contained in $M$, else the extremal distributions are ``nearly uniform'' representing an appropriately scaled convex combination of a uniform distribution on a set of size $\floor{C}$ and a disjoint point mass.  When $1 < C \leq 2$, the extreme points of $\mathcal{P}_C(M)$ are probability mass functions supported on exactly two points.  A more general proof of this result is given in \cite{MMX17:2}.  As we will not have use for the generality, we provide a simpler proof of this result and others of this subsection in the appendix \cite{BernoulliAppendix} to allow this article to be more self-contained.

%{Cyril: why calling the next statement a Corallary. This seems to me independent of the previous theorem?}

\begin{thm} \label{thm: Extreme point reduction}
    Let $m$ be a natural number and recall that $\llbracket  m\rrbracket = \{0,1,\dots,m\}$.  
    For $\alpha \in [0,\infty]$, a natural number $n$,  constants $C_1,\dots,C_n \leq m+1$, independent random variables $X_i$ with probability mass functions $f_{X_i} \in \mathcal{P}_{C_i}(\llbracket  m\rrbracket)$, it holds
    \begin{align} \label{eq: Extreme point reduction}
        H_\alpha(X_1 + \cdots + X_n) \geq \min_{Z \in \mathcal{E}} H_\alpha(Z_1 + \cdots + Z_n) ,
    \end{align}
    where $\mathcal{E}$ is the collection of all $Z = (Z_1, \dots, Z_n)$ such that $Z_i$ are independent variables with density $f_{Z_i} \in \mathcal{E}(\mathcal{P}_{C_i}( \llbracket m \rrbracket) )$. 
    %{Cyril: is $\alpha \in [0,\infty]$ here or any restriction on the range?} {\color{blue} James: I believe it holds for all $\alpha \in [0,\infty]$}
\end{thm}

\begin{remark}
We stress that the minimum in the right hand side of \eqref{eq: Extreme point reduction} is indeed a minimum and therefore is achieved.
\end{remark}

We restate the case that $\alpha = \infty$ below.

\begin{cor} \label{cor: min-entropy reduction}
For $X_1, \dots, X_n$ independent random variables taking values in a finite set $M$ such that $H_\infty(X_i) \geq \log C_i$, there exists $U_1, \dots, U_n$ independent such that $f_{U_i} \in \mathcal{E}(\mathcal{P}_{C_i}(M))$ and
\[
    H_\infty(X_1 + \cdots + X_n) \geq H_\infty( U_1 + \cdots + U_n).
\]
\end{cor}

%{Cyril: There is no proof of the above theorem. It seems to me that the proof of Theorem  \ref{cor: min-entropy reduction} is more or less contained in the proof of Theorem \ref{thm: Extreme point reduction}. There might ba anyway some confusion between the set $M$ and $\{0,1,\dots,m\}$. I am probably missing something... in particular, the assumption $H_\infty(X_i) \geq \log C_i$ must be useful somewhere...

%Now I see that in the next section this is Theorem \ref{cor: min-entropy reduction} that is used and not the corollary. I guess there is some cleaning to be done here.}

%{\color{blue} James: Light cleaning done here.  Changed Theorem to corollary, I believe it just a restatement of the result with $\alpha = \infty$.}

\subsection{Rearrangement}

In this section we define the notions of log-concavity and of rearrangement of functions on the integers  $\mathbb{Z}$, to be used later on.

\begin{defn} \label{def:log-concave}
A function $f \colon \mathbb{Z} \to [0,\infty)$ is log-concave when
\[
    f^2(n) \geq f(n+1) f(n-1).
\]
and $f(i) f(j) >0$ for $i < j$ implies $f(k) >0$ for $k \in [i,j]$.
\end{defn}

\begin{defn}
For a function $f: \mathbb{Z} \to [0,\infty)$ with finite support,
\[
    f = \sum_{i=0}^{n} a_i \mathbbm{1}_{\{ x_i\}}
\]
with $x_1 < x_2 < \cdots < x_n$ denote
\[
    f^\# = \sum_{i=0}^n a_i \mathbbm{1}_{\{i\}}.
\]
When $X_i$ are independent random variables with densities $f_i$, we denote by $X_i^{\#}$ a collection of random variables such that $X_i^\#$ has density $f_i^{\#}$.
\end{defn}

\subsection{Integers}

In this section, we will make use of a result due to Madiman-Wang-Woo \cite{MWW19} (see also \cite{MWW21}) that shows that the R\'enyi entropy power is somehow decreasing under rearrangement. More precisely, these authors prove that $f_1* \cdots * f_n$ is majorized by $f_1^\# * \cdots * f_n^\#$. We refer to \cite{MOA11:book} for background on majorization.

The next theorem follows from \cite{MWW19}, by applying the Schur concavity of R\'enyi entropy.

\begin{thm}[Theorem 1.4 \cite{MWW19}] \label{thm: Madiman majorization theorem}
For $\alpha \in [0,\infty]$, and $f_i$ are such that $f_i^\#$ are log-concave then
\[
     H_\alpha(f_1 * \cdots * f_n) \geq H_\alpha(f_1^\# * \cdots * f_n^\#).
\]
\end{thm}

The significance of the theorem for our pursuits here, is that it will reduce our investigations of a min-entropy power inequality to Bernoulli and uniform distributions.

\begin{cor} \label{cor: Bernoulli and Uniforms reduction}
    For $X_i$ independent variables with $\|f_{X_i}\|_\infty \leq 1/C_i$
    and with $C_i \in (1,2] \bigcup \cup_{i=3}^{\infty} \{i\}$ then
    \begin{align*}
        H_\infty(X_1 + \cdots +X_n) \geq H_\infty(Z_1 + \cdots + Z_n)
    \end{align*}
    where the variables $Z_i$ are independent and $f_{Z_i} \in \mathcal{E}(\mathcal{P}_{C_i}(\llbracket \floor*{C_i} \rrbracket))$. Moreover, $Z_i$ is Bernoulli when $C_i \in (1,2]$ and Uniform on $\{0,1,\dots, C_i-1\}$ when $C_i \in \cup_{i=3}^\infty \{i\}$.
\end{cor}
The proof is given in the appendix \cite{BernoulliAppendix}.

\subsection{Min Entropy Inequality}

The aim of this section is to prove the following Min Entropy power inequality which constitutes one of our main theorems.

\begin{thm}[Min-EPI] \label{thm: Min-EPI}
For independent integer-valued random variables $X_i$, the following holds
\begin{align*}
    \Delta_\infty(X_1 + \cdots + X_n) \geq \frac 1 {22} \sum_{i=1}^n \Delta_\infty(X_i).
\end{align*}
\end{thm}

In order to prove Theorem \ref{thm: Min-EPI}, we will use a comparison between the min-entropy and the variance (in both directions). Such a comparison is essentially known in the literature. The next result {holds} for all random variables and is  sharp for uniform distributions.

\begin{thm}[Bobkov-Chistyakov \cite{BC15:2, BMM20}] \label{thm: General min entropy upper bounded by variance}
For a discrete {random} variable $X$,
\begin{align} \label{eq: General min entropy upper bounded by variance}
    \Delta_\infty(X) \leq 12  \Var(X).
\end{align}
\end{thm}

To state the other direction, we need to introduce two definitions. First we say that 
an integer-valued random variable $X$ is log-concave if its probability mass function $f_X$ is log-concave, in the sense of Definition \ref{def:log-concave}. In other words,
$f_X^2(n) \geq f_X(n-1) f_X(n+1)$ holds for all $n$, and $f_X(k)f_X(n) > 0$ for $k < m < n$ implies $f_X(m) >0$. Next we define the notion of symmetry.

\begin{defn}[Symmetric Random variable]
    A real-valued random variable $X$ is symmetric when there exists $a$, such that
    \begin{align*}
        \mathbb{P}(X = a + x) = \mathbb{P}(X= a-x).
    \end{align*}
    holds for all $x \in \mathbb{R}$.
\end{defn}
Note that since we consider only $X$ integer-valued, $ a \in \frac 1 2 \mathbb{Z}$.  Also recall that both symmetry and log-concavity are preserved under independent summation.

\begin{thm}[Bobkov-Marsiglietti-Melbourne \cite{BMM20}] \label{thm: BMM on logconcave symm max versus variance}
    For an integer-valued, symmetric, log-concave {random} variable $X$,
    \begin{equation} \label{eq: BMM on logconcave symm max versus variance}
    \Delta_\infty(X) \geq 2  \Var(X).
    \end{equation}
\end{thm}

Strictly speaking, the result in \cite{BMM20} only covered the case that $X$ was symmetric about an integer point.  The argument used a majorization result from \cite{MT20} to reduce to a distribution on $\mathbb{Z}$ of the form $n \mapsto C p^{|n|}$ (which is resolved through direct computation), and the fact proven in \cite{BMM20}, that the variance is Schur-concave on the space of symmetric distributions.  For the convenience of the reader we cover the (missing) case that $X$ is symmetric about a point belonging to $\frac 1 2 + \mathbb{Z}$ in the appendix \cite{BernoulliAppendix}.

We are now in position to prove Theorem \ref{thm: Min-EPI}.

\begin{proof}[Proof of Theorem \ref{thm: Min-EPI}]
Given $X_1, X_2, \dots, X_n$ independent, we assume without loss of generality that for $i \leq k$, $\|f_{X_i} \|_\infty \geq 1/2$ and $i > k$ implies $\|f_{X_i}\|_\infty < 1/2$. %{Cyril: what is $M(X_i)$?}.  
By Corollary \ref{cor: min-entropy reduction},
\begin{align*}
    \Delta_\infty(X_1 + \cdots + X_n) &\geq \Delta_\infty(B_1 + \cdots + B_k + Z_{k+1} + \cdots + Z_n)
\end{align*}
where $B_i$ and $Z_j$ are all independent, the $B_i$ are Bernoulli satisfying $M(B_i) = M(X_i)$ and $Z_i$ is uniformly distributed on $\{1,2, \dots, n_i\}$ where $n_i$ is uniquely determined by $M(X_j) \in (\frac{1}{n_i + 1}, \frac 1 {n_i}]$.  Note that trivially,
\begin{align*}
    \Delta_\infty(X_1 + \cdots + X_n)
        \geq \max \{\Delta_\infty(B_1 + \cdots + B_k), \Delta_\infty(Z_1 + \cdots + Z_n) \} .
\end{align*}
Using the variance-min entropy comparisons above (\textit{i.e.}\ Inequalities  \eqref{eq: General min entropy upper bounded by variance} and \eqref{eq: BMM on logconcave symm max versus variance}) for Bernoulli random variables, 
\begin{align*}
    \Delta_\infty \left( \sum_i B_i \right)     \geq 
            2 \Var \left( \sum_i B_i \right) 
        = 
            2 \sum_i \Var(B_i) 
        \geq 
            \frac 1 6 \sum_i \Delta_\infty(B_i)
        =
            \frac 1 6 \sum_i \Delta_\infty(X_i) .
\end{align*}
Similarly, since $Z_1 + \cdots +Z_n$ is an independent sum of symmetric (about the point $\frac{n_i - 1}{2}$) log-concave variables, using the variance-min entropy comparisons, this time for symmetric log-concave variables,
\begin{align*}
    \Delta_\infty(Z_{k+1} + \cdots + Z_n)
        \geq
            2 \sum_j \Var(Z_j)
        \geq
            \frac 1 6 \sum_j \Delta_\infty(Z_i).
\end{align*}
Since $n_j \geq 2$
\begin{align*}
    \frac{\Delta_\infty(Z_j)}{\Delta_\infty(X_j)} 
        \geq 
            \frac{ n_j^2 -1}{(n_j +1)^2 -1}
        \geq
            \frac{2^2 - 1}{(2+1)^2-1}
        = 
            \frac 3 8.
\end{align*}
Thus it follows that
\begin{align*}
    \Delta_\infty(Z_{k+1} + \cdots Z_n)
        \geq 
            \frac 1 {16}\sum_{j = k+1}^n \Delta_\infty(Z_j).
\end{align*}
Finally,
\begin{align*}
    \Delta_\infty(X_1 + \cdots + X_n) 
        \geq
            \max \left\{ \frac 1 6 \sum_{i=1}^k \Delta_\infty(X_i), \frac 1 {16} \sum_{j = k+1}^n \Delta_\infty(X_j) \right\}
        \geq
            \frac 1 {22} \sum_{i=1}^n \Delta_\infty(X_i).
\end{align*}
where we use the fact that $\max(\alpha a, \beta b) \geq \frac{\alpha \beta}{\alpha + \beta}(a+b)$, valid for any non-negative $\alpha, \beta, a, b$.
\end{proof}

\subsection{Min-Entropy inequalities: Tightenings, and Reversals}

% Let us observe first that Theorem \ref{thm: Min-EPI} allows a $d$-dimensional extension.

% \begin{cor}
%     For $X_i$ independent $\mathbb{Z}^d$ random variables, with $\Delta_\infty(X) \coloneqq \|f_{X}\|_\infty^{-2} - 1$,
%     \begin{align*}
%         \Delta_\infty(X_1 + \cdots + X_n) \geq \frac{1}{22} \sum_i \Delta_\infty(X_i).
%     \end{align*}
% \end{cor}

% \begin{proof}
% By the arguments in \cite{MWW19}, these results can be immediately obtained by rearrangement to the integers. {Cyril: could you be more precise/specific. From the phrasing it seems that the result is obvious from \cite{MWW19} while we spent 2 pages of computations to prove the one dimensional case...}
% \end{proof}
% Note that as $X_i$ integer-valued random variables can be extended to $\mathbb{Z}^d$ by $\tilde{X}_i = (X_i,0, \dots, 0)$, unlike the continuous case, the discrete min-entropy power inequality derived here does not admit dimensional improvement in general.  
% {Cyril: in the continuous also one can consider $(X_i,0, \dots, 0)$. Where is the difference?}
 
For the min-entropy we may also give a reversal of min-EPI for Bernoulli sums.
\begin{thm}
For $X_i$ independent Bernoulli sums,
\begin{align*}
    \Delta_\infty(\sum_i X_i) \leq 6 \sum_i \Delta_\infty(X_i).
\end{align*}
\end{thm}

\begin{proof}
    By Theorems \ref{thm: General min entropy upper bounded by variance} and \ref{thm: Renyi entropy by variance} (see Inequality \eqref{eq:PB}),
    \begin{align*}
        \Delta_\infty(X_1 + \cdots + X_n) 
            \leq
                12 \ \Var\left(\sum_i X_i \right)
            =
                12 \sum_i \Var(X_i)
            \leq 
                6 \sum_i \Delta_\infty(X_i).
    \end{align*}
\end{proof}

In the case that $X_i$ are concentrated about a point, one can actually tighten the min-EPI beyond the $\Delta_\infty(\sum_i X_i) \geq \frac 1 6 \sum_i \Delta_\infty(X_i)$ that one would achieve through variance comparisons in the min-EPI reversals for Bernoulli sums, as we show in what follows.  We will need the following Lemma whose proof is suppressed to the appendix.

\begin{lem} \label{lem: Bernoulli exponential bound with min entropy power}
For $X$ a Bernoulli random variable, and $t \in [-\pi, \pi]$,
\begin{align}\label{eq: characteristic bounds by min entropy power}
    |\mathbb{E} e^{itX}| \leq e^{-\Delta_\infty(X) t^2/24}.
\end{align}
\end{lem}

Observe that, when $X$ is expanding the inequality at $t=0$ shows that the constant $1/24$ is optimal in the latter and we will show that this inequality can be used to derive a sharpening of the min-EPI for for Bernoulli sums. % improves upon bounds derived from min-entropy to variance comparisons a constant of $\frac{1}{3 \pi^2}$ could be obtained from the variance bound used previously and the min-entropy/variance comparison for Bernoulli random variables in Lemma \ref{lem: Delta versus variance}. 
%{Cyril: I understand that $1/24$ is optimal. But the discussion about $3\pi^2$ is not clear to me... Can you say more on this?}
%{\color{blue} James: good catch, it was not very clear.  I have tried to say less here.}

\begin{thm}\label{thm: small value min EPI}
For $X_i$ independent and integer-valued such that $\|f_{X_i} \|_\infty = c_i \geq \frac 1 2$,
\begin{align*}
    \Delta_\infty(X_1 + \cdots + X_n) \geq \frac{\pi^2}{36} \sum_{j=1}^n \Delta_\infty(X_i).
\end{align*}
\end{thm}

\begin{proof}
By Corollary \ref{cor: Bernoulli and Uniforms reduction} it suffices to prove the result when the $X_i$ are independent Bernoulli. By Lemma \ref{lem: Bernoulli exponential bound with min entropy power},
\begin{align*}
    \| \hat{f}_{X_i} \|_q^q 
        \leq 
            \frac{1}{2\pi} \int_{-\pi}^{\pi} e^{-\Delta_\infty(X_i) q t^2/24} dt
        =
            \frac{\int_0^{\sqrt{\pi^2 \Delta_\infty(X_i) q/12}} e^{-t^2/2} dt}{\sqrt{\pi^2\Delta_\infty(X_i) q/12}}
           .
\end{align*}
Thus applying Lemma \ref{lem: Fourier bound to fourier bound on sum} with $\Phi(q) = \frac{1}{\sqrt{\pi^2  q/12}}
           \int_0^{\sqrt{\pi^2q/12}} e^{-t^2/2} dt$, and $c_i = \Delta_\infty(X_i)$, we have
\begin{align*} 
   \| \hat{f}_{\sum_i X_i} \|_q^q \leq \Phi \left(q \sum_i \Delta_\infty(X_i) \right).
\end{align*}

% We proceed as in the variance comparison using independence and H\"older's inequality for $q_j \geq 1$ such that $\sum_j \frac 1 {q_j} = 1$ 
% \begin{align}
%     \|f_{\sum_j X_j}\|_\infty 
%         &\leq
%             \| \hat{f}_{\sum_j X_j} \|_1
%                 \\
%         &=
%             \| \prod_j \hat{f}_{X_j} \|_1
%                 \\
%         &\leq
%             \prod_j \left( \frac{1}{2 \pi} \int_{-\pi}^{\pi} |\mathbb{E} e^{itX_j}|^{q_j} dt \right)^{\frac 1 {q_j}}.
% \end{align}
% By Lemma \ref{lem: Bernoulli exponential bound with min entropy power} this gives,
% \begin{align}
%     \| f_{\sum_j X_j} \|_\infty
%         &\leq
%             \prod_j \left( \frac 1 \pi \int_0^{\pi} e^{-q_j\Delta_\infty(X) t^2/24} dt \right)^{\frac 1 {q_j}}
%                 \\
%         &=
%             \prod_j \left( \sqrt{\frac{12}{ \pi^2 q_j \Delta_\infty(X_j)}}  \int_0^{  \sqrt{ \pi^2 q_j \Delta_\infty(X_j)/12} }  e^{-t^2/2} dt\right)^{\frac 1 {q_j}}
%                 \\
%         &\leq
%              \left( \sqrt{\frac{12}{ \pi^2 \sum_j \Delta_\infty(X_j)}}  \int_0^{ \pi \sqrt{ \pi^2 \sum_j \Delta_\infty(X_j)/12} }  e^{-t^2/2} dt\right)
% \end{align}
By Hausdorff-Young, this gives
\begin{align*}
    \| f_{\sum_i X_i } \|_\infty 
        \leq 
            \| \hat{f}_{\sum_i X_i} \|_1 
        \leq
            \frac{\int_0^{\sqrt{\pi^2(\sum_j \Delta_\infty(X_j))/12}} e^{-t^2/2} dt}{\sqrt{\pi^2 (\sum_j \Delta_\infty(X_j)) /12}}.
\end{align*}
Applying Lemma \ref{lem: Gaussian integral bound} with $z = \sqrt{\pi^2 \sum_j \Delta_\infty(X_j)/12}$ this gives
\begin{align*}
    \| f_{\sum_j X_j} \|_\infty \leq \sqrt{\frac{1}{1 + \frac{\pi^2}{36} \sum_j \Delta_\infty(X_j)}}
\end{align*}
which yields,
\begin{align*}
    \Delta_\infty(\sum_j X_j) \geq \frac{\pi^2}{36} \sum_j \Delta_\infty(X_j).
\end{align*}
\end{proof}

Let us note that the largest constant $c$ such that $\Delta_\infty(\sum_i X_i) \geq c \sum_i \Delta_\infty(X_i)$ holds for any collection of independent $X_i$ is no larger than $\frac 1 2$, as can be seen by taking $X_1$ and $X_2$ to be iid Bernoulli with parameter $p = 1/2$
(and $\pi^2/36 \simeq 0.27 > \frac{1}{4}$). Note that one can alternatively apply Theorem \ref{thm: Renyi entropy by variance} and \ref{lem: Delta versus variance} to obtain a similar result $\Delta_\infty(X_1 + \cdots + X_n) \geq 2 \sum_{i=1}^n \Var(X_i) \geq \frac{1}{6} \Delta_\infty(X_i)$ at the expense of a constant.  Also note that applying the Bernoulli tightening to the proof of Theorem \ref{thm: Min-EPI} gives
\begin{align*}
    \Delta_\infty(X_1 + \cdots + X_n) 
        \geq 
            \max \left\{ \frac{\pi^2}{36} \sum_{i=1}^k \Delta_\infty(X_i), \frac 1 {16} \sum_{j= k+1}^n \Delta_\infty(X_i) \right\}
 \geq
            \frac{\sum_{i=1}^n \Delta_\infty(X_i)}{16 + 36/\pi^2} ,
\end{align*}
and an improvement to a constant $c = \frac{1}{16 + 36/\pi^2} > \frac{1}{20}$ in Theorem \ref{thm: Min-EPI}.

% \begin{cor}
% Suppose that $X_i$ are independent Bernoulli with parameter $p$, and $a_i$ coefficients in $\mathbb{R} - \{0\}$, then
% \begin{align}
%     \mathbb{P}( a_1 X_1 + \cdots + a_n X_n = x) \leq \frac{1}{6np(1-p)} \int_0^{\sqrt{6n p(1-p)}} e^{-t^2/2} dt
% \end{align}
% \end{cor}

% For $X$ a $\mathbb{Z}$-valued random variable, denote $f_X(n) = \mathbb{P}(X = n)$, $H_p(X) = (1-p)^{-1} \log \sum_n f_X^p(n)$, and $\hat{f}_X(t) = \mathbb{E} e^{itX}$.

% \begin{lem} \label{lem: L^p bounds on Bernoulli}
% If $X$ is a Bernoulli random variable with parameter $\theta$, then for $q > 0$,
% \begin{align}
%     \| \hat{f}_X \|_q^q \leq  \frac{1} {\sqrt{ 2 q \lambda}} \int_{0}^{\sqrt{2 q \lambda}} e^{-t^2/2} dt,
% \end{align}
% where $\lambda$ is twice variance of $X$, equal to $2 \theta(1-\theta)$.
% \end{lem}

\section{Littlewood-Offord Problem}  \label{sec: littlewood-offord problem}

In this section we apply the results above to an entropic generalization of the Littlewood-Offord problem.

As usual, we denote by the dot sign the usual scalar product in $\mathbb{R}^n$ so that $S_v=v \cdot B$ with $v=(v_1,\dots,v_n)$ and $B=(B_1,\dots,B_n)$.
We first present a generalization of \cite[Theorem 1.2]{Sin19} {(see also \cite{JK21})}.

\begin{lem} \label{lem: Entropic Problem reduction to Bernoulli}
Let $B=(B_1,\dots,B_n)$ such that $B_i$ are independent Bernoulli random variables, for $\alpha \in [0,\infty]$ and $v=(v_1,\dots,v_n) \in \mathbb{R}^n$ such that $v_i \neq 0$ for all $i$, with $S_v \coloneqq  v_1 B_1 + \cdots + v_n B_n$ it holds that
\begin{align*}
    H_\alpha( S_v) \geq  \max_{a \in \{-1,1\}^n} H_\alpha (S_a).
\end{align*}
\end{lem}

\begin{proof}
Considering $\mathbb{R}$ as a vector space and choose\footnote{To find such a map, when $n=1$ choose a Hamel Basis for $\mathbb{R}$ over $\mathbb{Q}$ that extends $\{v_1\}$, take $\Phi(v_1) =1$ and all other basis elements to $0$.  By induction, choose a linear map such that $\Phi(v_i) \neq 0$ for $i < n$, and $\Psi$ such that $\Psi(v_n) \neq 0$.  Then choose a map of the form $\lambda \Phi + \Psi$ for $\lambda \in \mathbb{Q}$ such that $\lambda \Phi(v_i) - \Psi(v_i) \neq 0$ for all $i$ and write $\tilde{T} = \lambda \Phi - \Psi$.  Since $\tilde{T}(v_i) = \frac{p_i}{q_i}$ for $p_i, q_i \in \mathbb{Z} - \{0 \}$, $T = q \tilde{T}$ for $q = \prod_i q_i$ yields such a map.} a linear function $T: \mathbb{R} \to \mathbb{Q}$ such that $T(v_i) \in \mathbb{Z} - \{ 0\}$.  Then
\begin{align*}
    H_p( S_v) \geq H_p( T(S_v))
\end{align*}
since  deterministic functions of a random variable decrease R\'enyi entropy. As $T(S_v) = T(v_1) B_1 + \cdots + T(v_n) B_n$, it suffices to consider integer coefficients. Assuming $v_i \in \mathbb{Z} - \{0\}$, take $a_i = sign(v_i)$ where $sign(x) = \mathbbm{1}_{(0,\infty)}(x) - \mathbbm{1}_{(-\infty,0)}(x)$, then $(v_1 B_1)^{\#} + \cdots + (v_n B_n)^{\#}$ has the same distribution as $S_a +m$ where $m \coloneqq \# \{i : v_i < 0 \}$.  

Applying Theorem \ref{thm: Madiman majorization theorem}, 
\begin{align*}
    H_\alpha(S_v) \geq H_\alpha((v_1 B_1)^{\#} + \cdots + (v_n B_n)^{\#}) = H_\alpha(S_a + m ) = H_\alpha(S_a).
\end{align*}
\end{proof}

\begin{thm} \label{thm: entropic Littlewood-Offord}
For $S_v = v_1 B_1 + \cdots + v_n B_n$, where $v_i \neq 0$, $B_i$ independent Bernoulli of variance $\sigma_i^2$ and $\alpha \geq 2$,
\begin{align*}
    H_\alpha(S_v) 
        \ \geq \ \log \left[ \frac{\sqrt{6 \sigma^2 \alpha'}}{\int_0^{\sqrt{6 \sigma^2 \alpha'}} e^{-t^2/2} dt} \right] 
        \ \geq \  \frac 1 2  \max \left\{ \log (1+ 2 \alpha' \sigma^2), 
  \log \left( \frac{12 \alpha' \sigma^2}{\pi } \right) \right\},
\end{align*}
where $\sigma^2 = \sum_i \sigma_i^2$.
\end{thm}

\begin{proof}
%The proof follows the lines of the proof of Corollary \ref{thm: Bernoulli sum inequality}.
%We may omit some details.

By Lemma \ref{lem: Entropic Problem reduction to Bernoulli}, it suffices to consider $v$ with $v_i = \pm 1$, and since $v_i B_i$ is the translation of a Bernoulli $X_i$ with the same variance as $B_i$
% \begin{align} \label{eq: Renyi entropy minimizing coefficients}
%     H_p(S_v) \geq \min_{a \in \{-1,1\}^n} H_p(S_a),
% \end{align}
% so it suffices to the result for coefficients that are $\pm 1$.  

% Since $1-B$ Bernoulli of parameter $1-p$ for $B$ Bernoulli $p$, $S_a = X_1 + \cdots + X_n - \ell$ where $X_i = B_i$ when $a_i =1$ and $X_i = 1-B_i$ when $a_i = -1$ and $\ell = |\{i: a_i = -1\}|$.  Thus applying Theorem \ref{thm: Renyi entropy by variance},
\begin{align*}
    H_\alpha( v_1 B_1 + \cdots + v_n B_n)
        =
            H_\alpha(X_1 + \cdots + X_n)
        \geq
            \log \frac{\sqrt{6 \sigma^2 \alpha'}}{\int_0^{\sqrt{6 \sigma^2 \alpha' }} e^{-t^2/2} dt}.
\end{align*}
The inequality follows from Theorem \ref{thm: Renyi entropy by variance}.  Applying Lemma \ref{lem: Gaussian integral bound} completes the proof.
%as it did in the proof of Theorem \ref{thm: Renyi entropy by variance}.
\end{proof}

{ When $v_i =1$ for all $i$ and the  $B_i$ are iid Bernoulli($\lambda$), then $$H_\alpha(S_v) = \frac 1 2 \log \left( 1 + 2 \alpha' \lambda n + o(\lambda), \right)$$ so the constant $2 \alpha'$ is optimal in the inequality $H_\alpha(S_v) \geq \frac 1 2 \log ( 1 + 2 \alpha' \sigma^2 )$, or equivalently
$
       \Delta_p (S_v) \geq 2 \alpha' \sigma^2.
$ for every $\alpha \geq 2$.
}

Let us relate Theorem \ref{thm: entropic Littlewood-Offord} to the usual Littlewood-Offord problem, that is determining upper bounds on 
%{ Cyril: Shall we say that finding an upper bound is a reformulation of the   Littlewood-Offord problem? In which case we should mention it. Otherwise we don't really understand why we are talking about LO problem.}
%{\color{blue} James: Agreed, I have mentioned the LO problem briefly below}
$
    Q(S,0) \coloneqq \max_x \mathbb{P}(S= x)
$
where $S \coloneqq \sum_{i=1}^n v_i B_i$ where $v_i \in \mathbb{R} \setminus \{0\}$ and  $B_i$ iid Bernoulli of parameter $p$, classically chosen with $p = 1/2$.   We recall the following question of Fox, Kwan, and Sauermann.
\begin{ques}[\cite{FKS19} Question 6.2]
For $(v_1, \dots, v_n) \in (\mathbb{R} \setminus \{0 \})^n$ and $B_1, \dots, B_n$ iid Bernoul\-li with some parameter $0 < p \leq 1/2$ and $S=v_1 B_1 + \cdots + v_n B_n$. What upper bounds (in terms of $n$ and $p$) can we give on the maximum point probability $Q(S,0) = \max_{x \in \mathbb{R}} \mathbb{P}(S=x)$?
\end{ques}
When $p = \frac 1 2$ bounds this is a reformulation of the classical problem, determining the number of subsums that fall in a given location \cite{LO43}.

\begin{lem} \label{lem: ELO Problem reduction to Bernoulli}
Let $B=(B_1,\dots,B_n)$ such that $B_i$ are independent Bernoulli random variables, for $v=(v_1,\dots,v_n) \in \mathbb{R}^n$ such that $v_i \neq 0$ for all $i$, it holds
\begin{align*}
    Q(v \cdot B,0) \leq \max_{a \in \{-1,1\}^n} Q(a \cdot B,0).
\end{align*}
\end{lem}

%{Cyril: if the above Lemma comes from \cite{juvskevivcius2019littlewood, singhal2019erdos}, we should explain why we give a proof. Is the proof below alternative? Or is it to make the paper self-contained?}

%{\color{blue} James: shorter proof!.  I also checked more carefully, and though a related result is provided in \cite{juvskevivcius2019littlewood}, it is not the same.}

\begin{proof}  This is Lemma \eqref{lem: Entropic Problem reduction to Bernoulli} in the case $\alpha = \infty$.
% Consider $\mathbb{R}$ as a vector space over $\mathbb{Q}$ and let $T: \mathbb{R} \to \mathbb{Q}$ be a linear map such that $T(v_i) \in \mathbb{Z} - \{0\}$ for all $i$.  
% \begin{align*}
%     \mathbb{P} \left( v_1 B_1 + \cdots + v_n B_n = x \right)
%         &\leq
%             \mathbb{P} \left( T(v_1 B_1 + \cdots +v_n B_n) = T(x) \right)
%                 \\
%         &=
%             \mathbb{P} \left( T(v_1) B_1 + \cdots + T(v_n) B_n = T(x) \right),
% \end{align*}
% % thus it suffices to consider the situation that the $v_i = \frac{p_i}{q_i} \in \mathbb{Q}$.  Multiplying by $q = \prod_{i=1}^n q_i$,
% % gives
% % \begin{align*}
% %     \mathbb{P}( v_1 B_1 + \cdots + v_n B_n = x)
% %         =
% %             \mathbb{P}( m_1 B_1 + \cdots + m_n B_n = y)
% %  \end{align*}  
% %  for $m_i \in \mathbb{Z}$.  
% Thus it suffices to consider the case of integer coefficients.  In this case, by Theorem \ref{thm: Madiman majorization theorem} applied with $\alpha = \infty$ %{Cyril: applied with $\alpha=\infty$?} {\color{blue} James: updated}
%  \begin{align*}
%      \mathbb{P}( v_1 B_1 + \cdots + v_n B_n = x)
%         &\leq
%             \max_{z \in \mathbb{Z}} \mathbb{P}( (v_1 B_1)^{\#} + \cdots (v_n B_1)^{\#} = z)
%                 \\
%         &=
%             \max_{z \in \mathbb{Z}} \mathbb{P}( a_1 B_1 + \cdots + a_n B_n  = z),
%  \end{align*}
%  where $a_i = sign(v_i) \coloneqq  \mathbbm{1}_{(0,\infty)}(v_i) - \mathbbm{1}_{(-\infty,0)}(v_i)$.
\end{proof}

Let us emphasize the following result, which follows from taking $\alpha = \infty$ in Theorem \ref{thm: entropic Littlewood-Offord}.

\begin{cor} \label{thm: Bernoulli sum inequality}
    For $v_i \in \mathbb{R} \setminus \{0\}$, $S_v = v_1 X_1 + \cdots + v_n X_n$ where $X_i$ are independent Bernoulli random variables with variance $\sigma_i$, and denoting by  $\sigma^2 = \sum_j \sigma_j^2$, then 
    \begin{align} \label{eq: bound for bernoulli sums}
        Q(S_v,0) \leq \frac 1 { \sqrt{6 \sigma^2}} \int_0^{ \sqrt{6 \sigma^2} } e^{-t^2/2} dt.
    \end{align}
\end{cor}

\begin{cor} \label{cor: Poisson regime}
    When $S_v = v \cdot B$ for $B = (B_1, \dots, B_n)$ for $B_i$ iid Bernoulli {random variables} of parameter $p$,  
    \begin{align*}
        Q(S_v,0) \leq \frac{1}{\sqrt{ 1 + 2 n p (1-p)} }.
    \end{align*}
\end{cor}

\begin{proof}
    Applying \ref{lem: Gaussian integral bound} to Corollary \ref{thm: Bernoulli sum inequality} while observing that $\sigma^2 = n p (1-p)$ gives the result.
    % By lower semi-continuity of the $X \mapsto Q(X,0)$, $Y = \lim_n S_n$, for a sequence of Poisson-Binomials $S_n$ with variance $\sigma_n^2$. Using Corollary \ref{thm: Bernoulli sum inequality} on $S_n$ it holds 
    % $$
    % Q(Y,0) \leq \liminf_n Q(S_n,0) \leq \liminf_n \frac{1}{\sqrt{1 + 2 \sigma^2_n}} = \frac{1}{\sqrt{1 + 2 \sigma^2 }},
    % $$
    % { Cyril: for the first inequality, can we exchange the max and the liminf?}
    % where the second inequality follows from \eqref{eq: The gaussian bound near zero} and the hypothesis \eqref{eq: bound for bernoulli sums}.  Note that  $Y_n$ are log-concave and hence $Y$ is a log-concave variable as well so there is no difficulty passing limits.
\end{proof}

 Corollary \ref{thm: Bernoulli sum inequality} and Corollary \ref{cor: Poisson regime} are sharp for any $n$ as can be seen by taking small variance Bernoulli and $v_i = 1$.  Moreover, since $X \mapsto Q(X,0)$ is lower semi-continuous with respect to the weak topology for $X$ taking values on $\mathbb{Z}$, that the inequality $Q(S,0) \leq \frac 1 { \sqrt{6 \sigma^2}} \int_0^{ \sqrt{6 \sigma^2} } e^{-t^2/2} dt$ holds when $S$ is Poisson of parameter $\lambda$, in which case $\sigma^2 = \lambda$ and we have the following Taylor expansions for $\lambda$ near zero
\begin{align*}
    Q(S,0) = e^{-\lambda} = 1 - \lambda + o(\lambda),
\end{align*}
while by Lemma \ref{lem: Gaussian integral bound} (with $\sigma^2=\lambda$)
$
    \frac 1 { \sqrt{6 \sigma^2}} \int_0^{ \sqrt{6 \sigma^2} } e^{-t^2/2} dt \leq 
    %\frac{1}{\sqrt{1 + 2 \lambda}} = 
    1- \lambda + o(\lambda) .
$
%$
%    z \mapsto \left(\frac {z}{ \int_0^z e^{-t^2/2} dt} \right)^2
%$
%has the expansion $1 + z^2/3$ about $0$.  In fact it can be proven through a  Calculus %exercise that for $z>0$,
%\begin{align} \label{eq: The gaussian bound near zero}
%    \frac{1}{z} \int_0^z e^{-t^2/2} dt \leq \frac{1}{\sqrt{1 + z^2/3}}.
%\end{align}
%Thus we have for small $\lambda$
%\begin{align}
%    1 - \lambda + o(\lambda) \leq \frac{1}{\sqrt{1 + 2 \lambda}} = 1- \lambda + o(\lambda),
%\end{align}
Thus in the ``Poisson regime'', when $S_v$ is a sum of many small variance Bernoulli, and Poisson approximation can be invoked, inequality \eqref{eq: bound for bernoulli sums} is tight. In the Gaussian regime, when the local limit theorem applies, for example for a sequence of iid Bernoulli random variables $S = X_1 + \cdots + X_n$ gives 
\begin{align*}
    Q(S,0) \leq \frac{1}{\sqrt{ 2\pi \sigma_S^2}} + o (1/\sqrt{n}).
\end{align*}
Meanwhile by integrating on the whole $[0,\infty)$, 
$$
\frac 1 {\sqrt{6 \sigma^2}} \int_0^{\sqrt{6 \sigma^2}} e^{-t^2/2} dt \leq \frac 1 {\sqrt{6 \sigma^2}} \int_0^{\infty} e^{-t^2/2} dt = \sqrt{ \frac {\pi}{12 \sigma^2}},
$$
so the bounds cannot be improved by more than a constant factor in the Gaussian regime.  In particular when the $X_i$ are iid Bernoulli with parameter $1/2$, then by Erd\"os's sharp solution (see \cite{Erd45}), to the Littlewood-Offord problem, $Q(S,0) \leq 2^{-n} \left( \begin{array}{c} n \\ \floor{n/2} \end{array} \right) \approx \sqrt{\frac{2}{\pi n}},$ while $ \sqrt{ \frac {\pi}{12 \sigma^2}} = \sqrt{ \frac {\pi}{3}} \frac{1}{\sqrt{n}}$ (and we are off by a factor of $\pi/\sqrt{6}\simeq 1.28$).  

.

\section*{Acknowledgements}
The authors thank Arnaud Marsiglietti for stimulating discussion and in particular suggesting the connection to the Littlewood-Offord question of \cite{FKS19}{, as well as an anonymous reviewer whose careful reading and suggestions have improved this article, and to whom Proposition \ref{prop: reviewer} is to be credited.} \\

{The last author was supported by the Labex MME-DII funded by ANR, reference ANR-11-LBX-0023-01 and ANR-15-CE40-0020-03 - LSD - Large Stochastic Dynamics, and the grant of the Simone and Cino Del Duca Foundation, France.}

%%%%%%%%%%%%%%%%%%%%%%%%%%%%%%%%%%%%%%%%%%%%%%
%% Supplementary Material, if any, should   %%
%% be provided in {supplement} environment  %%
%% with title and short description.        %%
%%%%%%%%%%%%%%%%%%%%%%%%%%%%%%%%%%%%%%%%%%%%%%
% \begin{supplement}
%\stitle{Supplemental}
%\sdescription{Omitted proofs included for completeness}
\begin{appendix}
\section*{Supplemental material}
In this appendix we will prove
Lemma 2.6, Lemma 2.9, Theorem 4.1 and 4.2, Corollary 4.8, Theorem 4.12,
and
 Lemma 4.14.

% \begin{proof}[Proof of Lemma \ref{lem: Nazarov Tricked}]
%     Recall that for $w$ with distribution function $W$, when $w^s \in L^1$ it has distribution function $t \mapsto W(t^{\frac 1 s} )$, so that by equation \eqref{eq: Distribution function integral formula} and a change of variables,
%     \[
%         \int_{\mathbb{R}} w^s = \int_0^\infty s t^{s-1} W(t) dt.
%     \]
%     Similarly when $w^s-v^s \in L^1$ one obtains an analogous formula for $\phi$, 
%     \[
%         \phi(s) = t_0^{-1} \int_0^\infty \left(\frac{t}{t_0} \right)^{s-1} (W(t) - V(t)) dt.
%     \]
%     Thus for $s > s'$, we have
%     \begin{align*}
%         \phi(s) - \phi(s')
%             &=  
%                 \phi(s) = t_0^{-1}  \int_0^\infty \left( \left(\frac{t}{t_0} \right)^{s-1} - \left(\frac{t}{t_0} \right)^{s'-1} \right) (W(t) - V(t)) dt.
%     \end{align*}
%     Note that this completes the proof since the integrand above is non-negative.
% \end{proof}

\begin{proof}[Proof of Lemma 2.6]
Observe that $w_\lambda(t) - v_\lambda(t)$ has no more than one zero on $(0,\pi)$
if and only if $H(t):=w_\lambda(t)^2 - v_\lambda(t)^2$ has no more than one zero on $(0,\pi)$. Taking a derivative, we see that for $\lambda > 0$ and small enough $t$
\begin{align*}
    H'(t) = \lambda \sin t - \frac{6 t\lambda e^{-3 \lambda t^2/\pi^2}}{\pi^2} > 0.
\end{align*}
Further we see that for $\lambda >0$, $H'(t) = 0$ iff $\lambda = \frac{\pi^2 \log \left( \frac{ 6t}{\pi^2 \sin t} \right)}{3 t^2}$.  Now we claim that the function
\begin{align*}
  (0,\pi) \ni  t \mapsto \lambda(t) = \frac{\pi^2 \log \left( \frac{ 6t}{\pi^2 \sin t} \right)}{3 t^2}
\end{align*}
is strictly increasing and hence one to one from $(0,\pi)$ into $(-\infty,\infty)$ (since $\lim_{t \to 0} \lambda(t)=-\infty$ and $\lim_{t \to \pi} \lambda(t)=+\infty$). Hence for fixed $\lambda$, there exists exactly one $t=t_\lambda$ such that $H'(t) = 0$.  Computing
\begin{align*}
    \lambda'(t) = \frac{ \pi^2}{3t^2} \left( -t \cot(t) + 2 \log \left( \frac{\pi^2 \sin(t)}{6t} \right) + 1 \right)
\end{align*}
and since $\lim_{t \to 0} \lambda'(t) = \infty$, it suffices to show that $\lambda'(t)$ has no zeros on $(0,\pi)$.
That is that
\begin{align*}
  f(t) := 1- t \cot(t) + 2 \log \left( \frac{\pi^2 \sin(t)}{6t} \right)
\end{align*}
has no zeros on $(0,\pi)$.  
Note that $\lim_{t \to 0} f(t) = 2 \log \left( \frac{\pi^2}{6} \right)  > 0$, so it is enough to show that $f$ is increasing on $(0,\pi)$. We have, for $t \in (0,\pi)$
\begin{align*}
f'(t)=1 + \cot (t) + \frac{t}{\sin^2 (t)} - \frac{2}{t} \qquad \mbox{and} \qquad     f''(t) = \frac{2}{t^2} - \frac{2 t \cos(t)}{\sin^3(t)} .
\end{align*}
Now we claim that $f''(t) >0$ on $(0,\pi)$ which is equivalent to saying that
\begin{align*}
    \sin^3(t) - t^3 \cos(t) > 0, \qquad t \in (0,\pi) .
\end{align*}
For $t \geq \pi/2$ this is immediate. For $t < \pi/2$ we use the Taylor series bounds $\sin (t) \geq t - t^3/6$ and $\cos(t) \leq 1 - t^2/2 + t^4/24$, 
\begin{align*}
     \sin^3(t) - t^3 \cos(t) 
        &\geq
            (t - t^3/6)^3 - t^3 ( 1- t^2/2 + t^4/24)
                \\
        &=
            t^7 (3-t)(t+3)/216
\end{align*}
which is clearly positive on $(0,\pi/2)$.  The claim is proved and hence $f'$ is increasing. Since $\lim_{t \to 0} f'(t)=1 >0$, we infer that $f$ is increasing on $(0,\pi)$ as expected.

Thus $H'$ is positive for small $t$ and has at most one zero, and thus it follows that $H$ has at most one $0$ on $(0,\pi)$ since $H(0) = 0$, and $H$ is increasing and then decreasing.
\end{proof}

\begin{proof}[Proof of Lemma 2.9]
The second term follows from $\int_0^z e^{-t^2/2} dt \leq \int_0^\infty e^{-t^2/2} dt = \sqrt{\pi/2}$, which implies
\begin{align*} %\label{eq: obvious bound for Gaussian integral}
\int_0^z e^{-t^2/2} dt/z \leq \sqrt{\frac{\pi}{2 z^2}} .
\end{align*}
The first term is more complicated.  It is enough to prove that $y \mapsto F(y) \coloneqq \frac{1}{\left(\int_{0}^y e^{-t^2/2}dt\right)^2} - \frac{1}{y^2}$ is non-decreasing on $(0,\infty)$. Indeed, this would imply for any $y >0$ that
$F(y) \geq \lim_{y \downarrow 0} F(y) = \frac{1}{3}$ which can be rephrased as the expected bound
\begin{equation} \label{eq:Paris}
\frac{ \int_0^z e^{-t^2/2} dt}{z} \leq \left( 1 + \frac{z^2}{3} \right)^{-1/2}.
\end{equation}
 
To prove that $F$ is non-decreasing, we take the derivative and obtain that
$$
F'(y) = \frac{2}{y^3 \left( \int_{0}^y e^{-t^2/2}dt\right)^3} 
\left(-y^3 e^{-y^2/2} + \left( \int_{0}^y e^{-t^2/2}dt\right)^3 \right) , \qquad y >0 .
$$
Set $G(y):= -y^3 e^{-y^2/2} + \left( \int_{0}^y e^{-t^2/2}dt\right)^3$, $y>0$, and observe that
$$
G'(y)= e^{-y^2/2} \left( -3y^2 + y^4 +3\left( \int_{0}^y e^{-t^2/2}dt\right)^2 \right) .
$$
Now, since $e^{-t^2/2} \geq 1- \frac{t^2}{2}$, we have
$\int_{0}^y e^{-t^2/2}dt \geq y - \frac{y^3}{6}$. Therefore, for $ y \in (0, \sqrt{6}]$ (so that
$y-\frac{y^3}{6} \geq 0$), it holds
$$
G'(y) 
\geq 
e^{-y^2/2} \left( -3y^2 + y^4 +3\left( y - \frac{y^3}{6}\right)^2 \right)
=
\frac{y^6e^{-y^2/2}}{12} > 0  .
$$
On the other hand, for $y \geq \sqrt{6}$, we observe that $-3y^2 + y^4 >0$ so that $G'(y)>0$ on $[\sqrt{6},\infty)$ and therefore on $(0,\infty)$. As a consequence
$G$ is increasing on $(0,\infty)$, and since $\lim_{y \downarrow 0} G(y) = 0$,
$F$ is non-decreasing on $(0,\infty)$ as expected. The limit as $y$ tends to zero is an easy consequence of the Taylor expansion $\int_0^y e^{-t^2/2} dt = y - y^3/6 + o(y^3)$, while \eqref{eq:Paris} is obtained directly from $F(y) \geq \frac 1 3$.
\end{proof}

\begin{proof}[Proof of Theorem 4.1]
We will prove the two inclusions ($\subset$, $\supset$) of the sets.

Given $f \in \mathcal{P}_C(M)$, if there exists $i \neq j$ such that $f(i), f(j) \in (0,\frac 1 C)$ then $g_1 = f + \varepsilon (\mathds{1}_{\{i\}} - \mathds{1}_{\{j\}})$ and $g_2 = f - \varepsilon (\mathds{1}_{\{i\}} - \mathds{1}_{\{j\}})$ are distinct elements of $\mathcal{P}_C(M)$ (for $\varepsilon$ small enough), and since $\frac{g_1 + g_2}{2} = f$, $f \notin \mathcal{E}(\mathcal{P}_C(M))$.  This proves that extreme points of $\mathcal{P}_C(M)$
have at most one value in $(0,\frac{1}{C})$ and therefore proves the first inclusion.

Conversely, consider $f$ such that there exists $i_o$ with $f(j) \in \{0, 1/C\}$ for all $j \neq i_o$ and suppose that $f = \frac{g_1 + g_2}{2}$ for some $g_1, g_2 \in \mathcal{P}_C(M)$.   For $j \neq i_o$, $f(j)$ is an extreme point of the interval $[0, \frac 1 C]$ and $\frac{g_1(j) + g_2(j)}{2} = f(j)$. Hence $g_1(j) = g_2(j) = f(j)$.  By the constraint that $f, g_1,$ and $g_2$ are probability mass functions we must have $f(i_o) = g_1(i_o) = g_2(i_o)$ as well and the second inclusion is proved.
\end{proof}

\begin{proof}[Proof of Theorem 4.2]
Note that $X_1 + \cdots + X_n$ is supported on $\llbracket nm \rrbracket$, and as a function of densities, %$\llbracket nm \rrbracket$ 
the map $f \mapsto H_\alpha(f)$ is continuous and quasi-concave in the sense that densities $f$ and $g$ satisfy $H_\alpha(\frac{f+g}{2}) \geq \min \{H_\alpha(f), H_\alpha(g) \}$.  Indeed, continuity is obvious since the probability distributions under consideration have finite support, and  quasi-concavity  follows from the expression $H_\alpha(f) = \alpha' \log \| f\|_\alpha$, and the convexity of $f \mapsto \|f \|_\alpha$ for $\alpha >1$ and its concavity for $\alpha < 1$.  Thus the continuity of the map from $\mathcal{P}_{C_1}(\llbracket  m\rrbracket) \times \cdots \mathcal{P}_{C_n} (\llbracket  m\rrbracket) \to  [0,\infty)$, given by $(f_1, \dots, f_n) \mapsto H_\alpha(f_1 * \cdots *f_n)$ is continuous since the convolution can be expressed as a polynomial of the terms of $f_i(k)$.  What is more, the map is coordinate quasi-concave, since convolution is coordinate affine, in the sense that $\frac{f_1 + g}{2}*f_2* \cdots * f_n = \frac{f_1*f_2*\cdots * f_n}{2} + \frac{g*f_2* \cdots *f_n}{2}$.  Thus the map $g \mapsto H_\alpha(g*f_{X_2} * \cdots * f_{X_n})$ is continuous and quasi-concave on $\mathcal{P}_{C_1}(\llbracket m \rrbracket)$. Since $\mathcal{P}_{C_1}(\llbracket m \rrbracket)$ is a compact, convex subset of $\mathbb{R}^{m+1}$, and since points of compact convex subsets can be written as convex combinations of their extreme points by Krein-Milman, 
there exists a minimizer $g_1 \in \mathcal{E}(\mathcal{P}_{C_1}( \llbracket m \rrbracket )$ such that
\begin{align*}
    H_\alpha(f_{X_1}* f_{X_2}* \cdots * f_{X_n}) \geq H_\alpha( g_1 * f_{X_2} *\cdots * f_{X_n})
\end{align*}
Iterating the argument gives the proof.
\end{proof}

\begin{proof}[Proof of Corollary 4.8]
    Let us first assume that the $X_i$ are all supported on a finite set $\llbracket m \rrbracket$ for some $m$. This implies that $C_i \leq |\llbracket m \rrbracket|=m+1$.
    By Theorem 4.2 there exists densities $g_i \in \mathcal{E}( \mathcal{P}_{C_i}( \llbracket m \rrbracket))$ such that 
    \begin{align*}
        H_\infty( f_{X_1} * \cdots * f_{X_n}) \geq H_\infty ( g_1 * \cdots * g_n).
    \end{align*}
    However, by the assumption that $C_i \in (1,2] \bigcup \cup_{i=3}^{m+1} \{i\}$ the $g_i$ either takes only two values, in which case $g_i^{\#}$ is a Bernoulli, or $g_i$ is a uniform distribution on $C_i$ values, in which case $g_i^{\#}$ is a uniform distribution on $\{0,1, \dots, C_i-1\}$.  In either case $g_i^{\#}$ is log-concave and Theorem 4.7, gives 
    \begin{align*}
        H_\infty( g_1* \cdots * g_n) \geq H_\infty( g_1^{\#} * \cdots * g_n^{\#}).
    \end{align*}
    Combining the two inequalities completes the proof when the $X_i$ have compact support.  The general inequality is an approximation argument.  Define an auxiliary function
    \begin{align*}
        f_i^{(m)}(n) = \begin{cases}
                        \min\{ f_{X_i}(n), \frac 1 {C_i} - \frac 1 m \} &\mbox{if } n \in \llbracket - m , m \rrbracket \\
                        0 &\mbox{else}.
                    \end{cases}.
    \end{align*}
    and from $f_i^{(m)}$ define a density
    \begin{align*}
        f_{\tilde{X}_i}^{(m)}(n) =  f_i^{(m)}(n) + \frac{ 1- \sum_{k = - m}^{m} f_i^{(m)}(k)}{2m+1} \mathbbm{1}_{\llbracket -m, m \rrbracket}(n) .
    \end{align*}
    The probability mass functions $f_{\tilde{X}_i}$ converge pointwise to $f_{X_i}$ with $m \to \infty$, are compactly supported and $\|f_{\tilde{X}_i}^{(m)}(n) \|_\infty \leq \|f_{X_i}\|_\infty \leq 1/C_i$.  Further pointwise convergence coincides with weak-convergence on $\mathbb{Z}$ by the Portmanteau theorem since all subsets of $\mathbb{Z}$ are closed and open. %{Cyril: a curiosity, in French Portmanteau is written porte-manteau and means coat hanger (not sure about the translation!)}  {\color{blue} James: An open problem!}
    Note that $f \mapsto H_\infty(f)$ is upper semi-continuous with respect to weak convergence since for $f_\beta \to f$ pointwise, $f(n) = \lim_\beta f_\beta(n) \leq \liminf_\beta \|f_\beta \|_\infty$.  Thus it follows that $\limsup_\beta H_\infty( f_\beta) \to H_\infty(f)$.  Taking the max over $n$ gives $\|f\|_\infty \leq \liminf_\beta \|f_\beta\|_\infty$ and hence 
    \begin{align*}
        \limsup_\beta H_\infty(f_\beta) \leq H_\infty(f).
    \end{align*}
    Since the map $(f_1, \dots, f_n) \mapsto f_1*\cdots*f_n$ corresponds to the mapping of a product space of measures $(\mu_1, \dots, \mu)$ to their product measure, a weak continuous operation, and composed with pushing forward $\mu_1 \otimes \cdots \otimes \mu_n$ under the continuous map $T(x) = x_1 + \cdots + x_n$ a weak-continuous mapping, $(f_1, \dots, f_n) \mapsto f_1*\cdots*f_n$ is a composition of weakly continuous functions and hence weakly continuous as well.  Thus, $(f_1, \dots, f_n) \mapsto H_\infty(f_1*\cdots*f_n)$ is upper semi-continuous and we have
    \begin{align*}
        H_\infty (f_{X_1}*\cdots * f_{X_n}) \geq  \limsup_\beta H_\infty( f_{\tilde{X}_1^{(m)}} *\cdots * f_{\tilde{X}_n^{(m)}}).
    \end{align*}
    Since the $ f_{\tilde{X}_1^{(m)}}$ are compactly supported $H_\infty( f_{\tilde{X}_1^{(m)}} *\cdots * f_{\tilde{X}_n^{(m)}}) \geq H_\infty(Z_1 + \cdots + Z_n)$.
\end{proof}

\begin{proof}[Proof of Theorem %\ref{thm: BMM on logconcave symm max versus variance}
4.12]
Suppose that $f$ is a log-concave density symmetric about a point $n - \frac 1 2$.  Note that if $f$ is supported on only two points, the inequality is true immediately, since by symmetry $f$ is a translation of  Bernoulli with parameter $1/2$.  Thus, assume that $f$ is supported on at least $4$ points so that $\|f\|_\infty < 1/2$ and observe that $\|f\|_\infty = f(n-1) = f(n)$.  Take $p = 1 - 2 \|f\|_\infty > 0$ and define a density $g$ by $g(n+k) = \|f\|_\infty p^k$ for $k \geq 0$ and $g(n+k) = \|f\|_\infty p^{-k-1} $ for $k < 0$.  Note that $g$ is symmetric log-concave density, satisfying $\|g\|_\infty  = \|f\|_\infty$ and $g \prec f$.  Translating $g(k) = g(n+k - 1/2)$, and $f(k) = f(n+k - 1/2)$ we obtain densities symmetric about $0$ taking values on $\frac 1 2 + \mathbb{Z}$.  It is straight forward to prove from the majorization that $X \sim g$ and $Y \sim f$ that $\mathbb{E} X^2 \geq \mathbb{E} Y^2$ and since both variables are centered, $ \Var(X) \geq  \Var(Y)$ while by definition $\Delta_\infty(X) = \Delta_\infty(Y)$.  Thus it suffices to prove the inequality for $X$ and $p \in (0,1)$.

In this case, with integer $k \geq 0$, $\mathbb{P}(X = \pm (\frac 1 2 +k)) = \frac{1-p}{2} p^k$, direct computations gives
\begin{align*}
    \Delta_\infty(X) &= \frac{4}{(1-p)^2} -1 ,
        \\
     \Var(X) &=
        (1-p) \sum_{k=0}^\infty \left( k + \frac 1 2 \right) p^k
            =
            \frac{ p^2+ 6p +1}{4(1-p)^2}.
\end{align*}
Thus the inequality $\Delta_\infty(X) \geq 2 \Var(X)$ is equivalent to proving $5 - 2p -3p^2 \geq 0$
% \begin{align*}
%     8 - 2(1-p)^2 \geq p^2 + 6p +1 
% \end{align*}
for $p \in [0,1]$.
% and since we wish to upper bound a convex function by a concave function it suffices to check the end points $p \in \{0,1\}$ so the inequality follows.  %The case that $X$ is symmetric about an integer is similar and given explicitly in \cite{BMM20}.
\end{proof}

\begin{proof}[Proof of Lemma 4.14]
Since both sides of (4.5) are invariant in the transformation of the parameter $p \mapsto 1-p$, we may assume $p \in [1/2,1]$ and compute explicitly with (2.4), it is equivalent to prove
\begin{align*}
    (1-p)^2 + p^2 + 2 p (1-p) \cos(t) \leq e^{- \left( \frac 1 {p^2} - 1\right) t^2/12}.
\end{align*}

Setting $q=1/p \in [1,2]$, the desired inequality is equivalent to proving that
$$
F(q):=q^2 e^{-\frac{q^2-1}{12}t^2} - 2(q-1)\cos(t) - (q-1)^2-1 \geq 0 .
$$
Our first aim is to prove that $F$ is concave on in the interval $[1,2]$ for any given $t \in [0,\pi]$. Observe that
$$
F'(q)= \left( 2q - \frac{q^3t^2}{6} \right) e^{-\frac{q^2-1}{12}t^2} -2 \cos(t) - 2(q-1)  ,
$$
and
$$
F''(q)=\left(2 - 2 e^{\frac{q^2-1}{12}t^2} + \frac{q^2t^2}{6}\left(-5 + \frac{q^2t^2}{6} \right) \right) e^{-\frac{q^2-1}{12}t^2} . 
$$
Given $t \in [0,\pi]$, set $r=\frac{q^2t^2}{6} \in [\frac{t^2}{6}; \frac{2t^2}{3}]$ and
$G(r):=2 - 2 e^\frac{r}{2}e^{-\frac{t^2}{12}} + r(-5+r)$ so that $F$ is concave on $[1,2]$ reduces to proving that $G$ is negative on 
$[\frac{t^2}{6}; \frac{2t^2}{3}]$. Observe that
$$
G'(r)=-e^\frac{r}{2}e^{-\frac{t^2}{12}} + 2r-5
\qquad \mbox{and} \qquad
G''(r)=-\frac{1}{2}e^\frac{r}{2}e^{-\frac{t^2}{12}} + 2 .
$$
We infer that $G''$ is decreasing on $[\frac{t^2}{6}; \frac{2t^2}{3}]$ and may change sign depending on the value of the parameter $t$.
We need to distinguish between two cases. 

(1) Assume first that $t \leq \sqrt{4\log 4}$. Then $G''(2t^2/3)= -\frac{1}{2}e^{\frac{t^2}{4}} + 2 \geq 0$. In that case we conldude that $G'' \geq 0$ on the whole interval and therefore that $G'$ is non-decreasing. Hence $G'(r) \leq G'(2t^2/3)=-e^{\frac{t^2}{4}} + \frac{4}{3}r^2-5$.
It is easy to see that the mapping $[0,\infty) \ni t  \mapsto H(t):=-e^{\frac{t^2}{4}} + \frac{4}{3}r^2-5$ is increasing on $[0, \sqrt{4 \log \frac{16}{3}}]$ so that, for $t \in [0, \sqrt{4\log 4}]$, $H(t) \leq H(\sqrt{4\log 4})=-9 + \frac{16}{3}\log (4) \simeq -1.6 \leq 0$.
Therefore $G'(r) \leq 0$ and hence $G(r) \leq G(t^2/6)=\frac{t^2}{6}(-5+\frac{t^2}{6}) \leq 0$ since $t \in [0,\sqrt{4\log 4}]$. 
As an intermediate  conclusion we proved that $G \leq 0$ in case (1).

(2) Assume now that $t \geq \sqrt{4\log 4}$. Then $G''$ changes sign. Namely, since $G''(t^2/6)=3/2 \geq 0$ and 
$G''(2t^2/3)= -\frac{1}{2}e^{\frac{t^2}{4}} + 2 \leq 0$, $G''$ is non-negative on $[\frac{t^2}{6}, r_o]$ and non-positive on $[r_o,\frac{2t^2}{3}]$, with $r_o:=2(\log 4 + \frac{t^2}{12})$. It follows that, for $t \leq \pi$ and $r \in [\frac{t^2}{6}; \frac{2t^2}{3}]$,
$G'(r) \leq G'(r_o)=4(\log 4 + \frac{t^2}{12})-9 \leq 4\log 4 + \frac{\pi^2}{3} - 9 \simeq -0.16 \leq 0$.
We conclude that $G$ is non-increasing and therefore that $G(r) \leq G(t^2/6)=\frac{t^2}{6}\left(-5+\frac{t^2}{6}\right) \leq 0$
since $t \leq \pi$. As a conclusion we proved that $G \leq 0$ in case (2) and therefore in any case. This shows that $F$ is concave.

Now $F$ being concave, $F \geq 0$ is a consequence of the fact that $F(1)=0$ and 
$F(2)=4e^{-\frac{t^2}{4}}-2\cos (t) - 2 \geq 0$.
To see the latter, one can observe that $2\cos (t) + 2 = 4 \cos^2(t/2)$ so that $F(2) \geq 0$ is equivalent to saying that
$e^{t^2/4} \geq \cos^2(t/2)$ on $[0,\pi]$ which in turn is equivalent to saying that
$e^{u^2/2} \geq \cos (u)$ for any $u=\frac{t}{2} \in [0, \frac{\pi}{2}]$. Taking the logarithm, we end up with proving that
$I(u):=-\frac{u^2}{2}-\log \cos (u) \geq 0$ on $[0, \frac{\pi}{2}]$. Since $I'(u)=-u+ \tan(u) \geq 0$ the desired conclusion immediately follows. This ends the proof of the Lemma.
\end{proof}

%\section{???}
%
%\section{???}
%
\end{appendix}
% \end{supplement}

%%%%%%%%%%%%%%%%%%%%%%%%%%%%%%%%%%%%%%%%%%%%%%%%%%%%%%%%%%%%%
%%                  The Bibliography                       %%
%%                                                         %%
%%  imsart-number.bst  will be used to                     %%
%%  create a .BBL file for submission.                     %%
%%                                                         %%
%%  Note that the displayed Bibliography will not          %%
%%  necessarily be rendered by Latex exactly as specified  %%
%%  in the online Instructions for Authors.                %%
%%                                                         %%
%%  MR numbers will be added by VTeX.                      %%
%%                                                         %%
%%  Use \cite{...} to cite references in text.             %%
%%                                                         %%
%%%%%%%%%%%%%%%%%%%%%%%%%%%%%%%%%%%%%%%%%%%%%%%%%%%%%%%%%%%%%

\bibliographystyle{plain}
\bibliography{bibibi}

\end{document}